\documentclass{amsart}
\usepackage{amssymb,latexsym}
\usepackage{amsmath,amsfonts,amsthm}
\usepackage[dvips]{graphicx}
\usepackage[all]{xy}
\newtheorem{theorem}{Theorem}[section]
\newtheorem{prop}[theorem]{Proposition}
\newtheorem{lemma}[theorem]{Lemma}
\newtheorem{remark}[theorem]{Remark}
\newtheorem{question}[theorem]{Question}
\newtheorem{definition}[theorem]{Definition}
\newtheorem{cor}[theorem]{Corollary}

\def\co{\colon\thinspace}
\def\delbar{\bar{\partial}}

\def\ep{\epsilon}
\begin{document}
\title[Floer homology and twisted geodesic flows]{Floer homology in disc bundles and symplectically twisted geodesic flows}
\author{Michael Usher}
\address{Department of Mathematics, University of Georgia, Athens, GA 30602}
\email{usher@math.uga.edu}
\begin{abstract}  We show that if $K\co P\to\mathbb{R}$ is an autonomous Hamiltonian on a symplectic manifold $(P,\Omega)$ which attains $0$ as a Morse-Bott nondegenerate minimum along a symplectic submanifold $M$, and if $c_1(TP)|_M$ vanishes in real cohomology, then the Hamiltonian flow of $K$ has contractible periodic orbits with bounded period on \emph{all} sufficiently small energy levels.  As a special case, if the geodesic flow on $T^*M$ is twisted by a symplectic magnetic field form, then the resulting flow has contractible periodic orbits on all low energy levels.  These results were proven by Ginzburg and G\"urel when $\Omega|_M$ is spherically rational, and our proof builds on their work; the argument involves constructing and carefully analyzing at the chain level a version of filtered Floer homology in the symplectic normal disc bundle to $M$.  
\end{abstract}
\maketitle
\section{Introduction}
Consider a symplectic manifold $(P,\Omega)$ containing a closed, connected symplectic submanifold $M$, with $2m=\dim M$, $2n=\dim P$, and $r=n-m$.  In recent years, there has been significant interest in the following question: \begin{question} \label{question} If $K\co P\to [0,\infty)$ is a proper smooth function with $K^{-1}(\{0\})=M$, must it be the case that the Hamiltonian vector field of $X_K$ (which in our convention is given by $\iota_{X_K}\Omega=dK$) has periodic orbits on all regular energy levels $K^{-1}(\{\rho\})$, provided that $\rho>0$ is sufficiently small?\end{question}

Already in the case that $M=\{pt\}$, V. Ginzburg and B. G\"urel show in \cite{GG1} that their negative resolution of the Hamiltonian Seifert conjecture (\cite{Gi}, \cite{GG0}) implies that the answer to Question \ref{question} is negative for some smooth Hamiltonians when $n\geq 3$ and for some $C^2$-Hamiltonians when $n\geq 2$ unless the requirements are weakened in some way.

One way of weakening the requirements is to require existence of periodic orbits only on a suitably large set of low energy levels, where ``large'' might be interpreted as meaning either ``dense'' or ``full measure.''  A considerable amount of work was done on this version of the problem (\emph{e.g.}, in  \cite{GK1}, \cite{Ke1}, \cite{GK2},  \cite{CGK}, \cite{GG1}, \cite{Ke2}), 
 culminating in results of L. Macarini \cite{M} and G. Lu \cite{L} which show that, for some $\rho_0>0$, it is the case that the level surfaces $K^{-1}(\{\rho\})$ contain periodic orbits for Lebesgue-almost-every regular value $\rho\in [0,\rho_0]$. Further, these periodic orbits are all contractible within a tubular neighborhood of $M$.

Another way of weakening the requirements of Question \ref{question} is to retain the requirement that orbits exist on all low energy levels, but to constrain the form of the function $K$ in some way.  It is this version of the question that we shall consider in this paper, with $K$ required to attain a Morse--Bott nondegenerate minimum along $M$ (in other words, the Hessian of $K$ is to restrict nondegenerately to the normal bundle of $M$).  An interesting special case of this (which historically served as much of the motivation for Question \ref{question}) is the following.  

Where $M$ is a closed Riemannian manifold, the motion of a particle of unit mass and unit charge in a magnetic field on $M$ is modeled by setting $P=T^*M$ and letting $\Omega=\Omega_{\sigma}=\omega_{can}+\tau^*\sigma$ where $\omega_{can}$ is the standard symplectic form on $T^*M$, $\tau\co T^*M\to M$ is the bundle projection, and $\sigma$ is a closed $2$-form on $M$ which represents the magnetic field.  The phase-space trajectory of the particle is then given by the Hamiltonian vector field $X_K$ of the standard kinetic energy Hamiltonian $K(q,p)=\frac{1}{2}|p|^2$.  Thus, $M$ is a symplectic submanifold of $(T^*M,\Omega_{\sigma})$ (so that this fits into the framework of Question \ref{question}) if and only if the magnetic field $2$-form $\sigma$ on $M$ is symplectic.  Of course, the case $\sigma=0$ just corresponds to the geodesic flow on $M$; accordingly the Hamiltonian flow of $K$ on $(T^*M,\Omega_{\sigma})$ is sometimes called the $\sigma$-twisted geodesic flow.  The search for periodic orbits of twisted geodesic flows  was initiated by V. Arnol'd (for the case $M=T^2$) in \cite{A}, and has continued in, \emph{e.g.}, \cite{G87}, \cite{Gsurv}, \cite{G96}, \cite{Lu0}, \cite{M0}, \cite{C}, \cite{CMP}, \cite{FS}, \cite{Pa}, \cite{Schlenk}, \cite{GG2}. In particular, it is shown by F. Schlenk in \cite{Schlenk} that there is $\rho_0>0$ such that $X_K$ has contractible periodic orbits on Lebesgue-almost-every energy level $\rho\in [0,\rho_0]$, provided merely that $\sigma$ does not vanish identically (thus in Schlenk's result $\sigma$ need not be symplectic, putting this result in a somewhat different category than those of \cite{L},\cite{M}).  

Note that choosing a compatible almost complex structure on $(P,\Omega)$ makes $TP$ into a complex vector bundle, which in particular has Chern classes $c_k(TP)\in H^{2k}(P;\mathbb{Z})$. 
Our main result is the following. 
\begin{theorem}\label{main} Let $M$ be a closed connected symplectic submanifold of the symplectic manifold $(P,\Omega)$, and let $\nu M$ be a tubular neighborhood of $M$.  Suppose that $K\co P\to [0,\infty)$ is a smooth function with $K^{-1}(\{0\})=M$, such that $K$ attains a Morse--Bott nondegenerate minimum along $M$.  Assume furthermore that $c_1(TP)|_{\nu M}$ represents a torsion class in $H^2(\nu M;\mathbb{Z})$.  Then there are $\rho_0>0$ and $T_0>0$ such that, for \emph{every} $\rho\in (0,\rho_0)$, the Hamiltonian flow of $K$ has a periodic orbit on $K^{-1}(\{\rho\})$ which is contained in and contractible in $\nu M$ and has period no larger than $T_0$. 
\end{theorem}

Since $c_1(T(T^*M))=0$, it immediately follows that:
\begin{cor}\label{magnet} If $\sigma\in \Omega^2(M)$ is symplectic, there are $\rho_0, T_0>0$ such that, for every $\rho\in (0,\rho_0)$ the Hamiltonian flow of the function $K(q,p)=\frac{1}{2}|p|^2$ on $(T^*M,\Omega_{\sigma})$ has a periodic orbit on $K^{-1}(\{\rho\})$ which is contractible in $T^*M$ and has period no larger than $T_0$.
\end{cor}

The restriction to sufficiently small energy levels is necessary: as is explained for instance in \cite{Gsurv}, already in the case where $M$ is a hyperbolic surface with area form given by $\sigma$, there is an energy level $c_0$ such that the magnetic flow has no contractible periodic orbits of energy larger than $c_0$, and no periodic orbits at all on the energy level $c_0$.

Theorem \ref{main} was proven in \cite{GG2} in the special case that $\Omega|_M$ is spherically rational (\emph{i.e.}, in the case that $\{\int_{S^2}u^*\Omega|u\co S^2\to M\}$ is a discrete subgroup of $\mathbb{R}$); in broad outline, our approach is similar to that of \cite{GG2}.  An important ingredient in the proof of the corresponding theorem in \cite{GG2} is a result (Propositions 3.1 and 3.2 of \cite{GG2}) which, in certain situations which include the case where $c_1(TP)|_{\nu M}$ is torsion, allows one to use grading information in Floer homology to bound the period of a periodic orbit.  As such, in order to obtain existence of periodic orbits on all sufficiently low energy levels, in these situations it is enough to find periodic orbits on a dense set of sufficiently low energy levels corresponding to a fixed Floer homological grading, since the periodic orbits so obtained then have bounded period and so the Arzel\`a-Ascoli theorem can be used to obtain periodic orbits
  on all low energy levels, with bounds on their period.  

Given this ingredient from \cite{GG2}, the main contribution of the present paper to the proof of Theorem \ref{main} is the following.
\begin{theorem}\label{main2}Let $M$ be a closed connected symplectic submanifold of the symplectic manifold $(P,\Omega)$, and let $\nu M$ be a tubular neighborhood of $M$.  Suppose that $K\co P\to [0,\infty)$ is a smooth function with $K^{-1}(\{0\})=M$, such that $K$ attains a Morse--Bott nondegenerate minimum along $M$.  Suppose also that either \begin{itemize} \item[(i)] $c_1(TP)|_{\nu M}$ represents a torsion class in $H^2(\nu M;\mathbb{Z})$, or
\item[(ii)] There is a Morse function on $M$ having no critical points of index $1$.\end{itemize}
Then there is $\rho_0$ with the following property.  If $0<\rho<\rho_0$ then for any sufficiently small $\ep>0$ the Hamiltonian flow of $K$ has a periodic orbit $\gamma\co \mathbb{R}/\tau\mathbb{Z}\to P$ of some period $\tau>0$, whose image is contained in $\nu M\cap K^{-1}((\rho-\ep,\rho+\ep))$ and is contractible in $\nu M$.  Furthermore,  there is a disc $w\co D^2\to \nu M$ having $w(e^{2\pi i t/\tau})=\gamma(t)$ $(0\leq t\leq \tau$) such that $[\gamma,w]$ has Salamon--Zehnder index $\Delta([\gamma,w],K)$ which satisfies $-2r\leq \Delta([\gamma,w],K)\leq 2m+1$. 
\end{theorem}
The notation $\Delta([\gamma,w],K)$ should be understood as follows.  Choose a symplectic trivialization of $w^*TP$ over $D^2$, and for $0\leq t\leq \tau$ let $\phi_{K}^{t}$ denote the time-$t$ flow of the Hamiltonian vector field $X_K$.  The linearizations at $\gamma(0)$ of $\phi_{K}^{t}$ are then represented with respect to our trivialization by a path $\{\Psi(t)\}_{0\leq t\leq \tau}$ of symplectic matrices, and $\Delta([\gamma,w],K)$ is the Salamon--Zehnder index (which was originally defined in \cite{SZ}, where it is denoted $\Delta_{\tau}(\Psi)$ and referred to as the ``mean winding number'') of this path $\Psi$.  Various relevant properties of $\Delta$ are reviewed in Section 2 of \cite{GG2}.
We should point out that since we define $X_K$ by $\iota_{X_K}\omega=dK$ instead of $\iota_{X_K}\omega=-dK$ as is done in \cite{GG2}, our Hamiltonian vector fields have periodic orbits which are related to those in \cite{GG2} by time reversal, as a result of which the Salamon--Zehnder indices $\Delta([\gamma,w],K)$ of these orbits have opposite sign.

  Except for the statement about the  Salamon--Zehnder index, Theorem \ref{main2} is weaker than the previously-mentioned results of \cite{L},\cite{M}.  However, as we shall shortly see, the information about the Salamon--Zehnder indices enables one to deduce Theorem \ref{main} from Theorem \ref{main2} in the situation of Case (i) above.  We include Case (ii) in Theorem \ref{main2}  because it illustrates the broad applicability of our method and its proof requires only a brief digression (in the proof of Proposition \ref{p1}) from the proof of Case (i).  Note that the condition that $M$ admit a Morse function with no index-one critical points obviously implies that $\pi_1(M)=0$, and in fact if $\dim M\neq 4$ this condition is equivalent to requiring that $\pi_1(M)=0$ (see Theorem 8.1 of \cite{Mil}).  Whether or not this remains true when $\dim M=4$ is an open question as of this writing.

Incidentally, the Morse--Bott assumption on $K$ plays a fairly modest role in the proof of Theorem \ref{main2}; it facilitates somewhat the construction of the functions $f_a,f_b$ of Section 4, but functions with the same essential properties could be constructed with some additional effort for many other classes of $K$.  However, the Morse--Bott assumption is vital in the proof that Theorem \ref{main2} implies Theorem \ref{main}, due to its use in Proposition 3.2 of \cite{GG2}.

\begin{proof}[Proof of Theorem \ref{main}, assuming Theorem \ref{main2}] Since $c_1(TP)|_{\nu M}$ is torsion, $c_1(TP)$ is represented in real cohomology by a $2$-form which vanishes throughout $\nu M$.  Let $\gamma$ be any of the  periodic orbits produced by Theorem \ref{main2}, and denote its period by $T_{\gamma}$. Then since the disc $w$ is contained in $\nu M$, Proposition 3.2 of \cite{GG2} gives that (possibly after shrinking $\rho_0$, and taking into account that our differing conventions on the signs of Hamiltonian vector fields results in a sign reversal for the Salamon--Zehnder index $\Delta([\gamma,w],K)$), there are $\gamma$-independent constants $a,c>0$ with\[  T_{\gamma}\leq\frac{1}{a}(-\Delta([\gamma,w],K)+c).\]  Thus given $\rho< \rho_0$, where $T_0=\frac{1}{a}(2r+c)$ it follows that for each $\ep\in (0,\rho_0-\rho)$ $X_K$ has a contractible-in-$\nu M$ periodic orbit of period at most $T_0$ in $K^{-1}((\rho-\ep,\rho+\ep))$.  The Arzel\`a-Ascoli theorem applied to the orbits so obtained from a sequence $\ep_k\searrow 0$ then shows that $X_K$ has a contractible-in-$\nu M$ periodic orbit of period at most $T_0$ in $K^{-1}(\{\rho\})$.\end{proof}

Now Theorems \ref{main} and \ref{main2} depend only on the behavior of $K$   in a (sufficiently small) tubular neighborhood $\nu M$ of $M$. The Weinstein neighborhood theorem (see, \emph{e.g.}, Theorem 3.30 of \cite{MS0}) asserts that if $(P_1,\Omega_1)$ and $(P_2,\Omega_2)$ both contain $M$ as a symplectic submanifold, with $\Omega_1|_M=\Omega_2|_M$ and with the symplectic normal bundles to $M$ in $P_1$ and $P_2$ isomorphic as symplectic vector bundles, then there are tubular neighborhoods $\nu_i \subset P_i$ of $M$ in $P_i$ ($i=1,2$) which are symplectomorphic by a symplectomorphism restricting to $M$ as the identity.  Where $E\to M$ denotes an arbitrary sympelctic vector bundle over $M$ (which of course necessarily admits complex vector bundle structure and a compatible  Hermitian metric),  at the beginning of Section 2 we shall, for suitably small $R>0$, equip the radius $R$ disc bundle $E(R)$ with a symplectic form $\omega$ which restricts to the zero section $M$ as an arbitrary given symplectic form $\omega_0$.  Given $(P,\Omega)$ as in the statement of Theorem \ref{main2}, use for $E$ the symplectic normal bundle to $M$ in $P$ and use for $\omega_0$ the restriction $\Omega|_M$.  A tubular neighborhood of $M$ in $P$ is then symplectomorphic to a tubular neighborhood of $M$  in $E(R)$, and the Morse-Bott condition for $K$ is obviously preserved via this symplectomorphism, so to prove Theorem \ref{main2} it is enough to prove it when $(P,\Omega)$ is the symplectic disc bundle $(E(R),\omega)$ and $M$ is the zero-section.  The rest of the paper is devoted to this latter task.

The proof uses a version of (filtered) Floer homology for the disc bundle $E(R)$.  Since $E(R)$ is not closed, there are subtleties involved in the definition this Floer homology, because of a need for compactness results.  As far as we know, the only case in which such Floer homology groups have been constructed in the literature without any constraint on the base $M$ is when $E$ satisfies a negative curvature hypothesis (this hypothesis allows one to use a maximum principle to obtain compactness; see \cite{O}); however such a hypothesis is not natural in our context.  In Section 2, we are able to address the compactness problem because the only Floer homology groups we need are groups of the form $HF^{[a,b]}(H)$ where $H$ is a Hamiltonian which behaves in a certain standard way outside a small neighborhood of the zero section $M$, and where the ``action window'' $[a,b]$ is small.  Thus all of the cylinders which one needs to use in the definition of the Floer complex have low energy; we prove a result (Theorem 
 \ref{c0}) showing roughly speaking that, for the Hamiltonians which we consider, a Floer connecting orbit with small energy stays entirely within a small neighborhood of the zero section.  (The main subtlety here is that the Hamiltonians can behave quite wildly very close to the zero section, so we need our constants to depend only on the behavior of the Hamiltonians away from the zero section.)  This enables us to define Floer groups $HF^{[a,b]}(H)$ in a fairly standard way.  Similar constructions are carried out in \cite{GG2} in the spherically rational case when the action interval $[a,b]$ does not intersect the image of $[\omega]$ on $\pi_2(E(R))$; there the compactness results are somewhat easier because in that context one can use certain compactly supported Hamiltonians in place of the Hamiltonians that we use.

Where $H$ is a reparametrized version of the Hamiltonian $K$ that we are interested in and where $H'$ is a small nondegenerate perturbation of $H$, the strategy is then to use a commutative diagram \begin{equation}\label{introdiag} \xymatrix{HF^{[a,b]}_{*}(F_0)\ar[rr]^{\Psi_{F_0}^{F_1}}\ar[dr]_{\Psi_{F_0}^{H'}}& & HF^{[a,b]}_{*}(F_1)\\ & HF^{[a,b]}_{*}(H')\ar[ur]_{\Psi_{H'}^{F_1}} & },\end{equation} where $F_0\leq H'\leq F_1$ and $F_0$ and $F_1$ are Hamiltonians with comparatively easy-to-understand Floer complexes, in order to obtain information about $HF^{[a,b]}(H')$ and hence about the periodic orbits of $H'$.  $F_0$ and $F_1$ are perturbed versions of Hamiltonians which depend only on the distance from the zero section, and Section 3 is concerned with learning about the Floer homologies of such Hamiltonians.  The crucial results in this direction are Lemma \ref{toprest} and its Corollary \ref{bigcor}, which lead to significant topological restrictions on the cylinders that are involved in the boundary operators of the $[a,b]$-Floer complexes of $F_0$ and $F_1$ and in the chain map relating the complexes when $b-a$ is sufficiently small; these topological restrictions are the 
 primary factors that enable us to deal with the case that $\Omega|_M$ is not spherically rational.  In spirit, what makes these restrictions possible is that the Hamiltonian vector fields of the Hamiltonians  considered in Section 3 are very nearly vertical, so that the cylinders $u\co \mathbb{R}\times (\mathbb{R}/\mathbb{Z})\to E(R)$ that are involved in the boundary operators and chain maps have $\pi\circ u$ nearly pseudoholomorphic (where $\pi\co E(R)\to M$ is the disc bundle projection).  Thus requiring $u$ and hence $\pi\circ u$ to have low energy restricts the topology of $\pi\circ u$, and hence also that of $u$ since $\pi$ induces an isomorphism on $\pi_2$.

In  Section 4 we fairly explicitly construct the Hamiltonians denoted above by $F_0$ and $F_1$ (and in Section 4 by $F_{b}^{\ep}$ and $F_{a}^{\ep}$), and then we leverage the results of Section 3 to prove Theorem \ref{main2}.  Here is an algebraic summary of the argument, with the geometry underlying the algebra deferred to Sections 3 and 4.  For a judiciously chosen action interval $[a,b]$ (equal to $[c(\rho),d(\rho)]$ in the notation of Section 4), in grading $2r$ one has, for $i=0,1$, $CF_{2r}^{[a,b]}(F_i)=\mathbb{Z}_2\oplus N_i$, where the $\mathbb{Z}_2$ is generated by a ``fiberwise-capped'' (in the terminology of Section 3) periodic orbit $x_i$, while the (typically infinitely-generated) summand $N_i$ is generated by periodic orbits with a capping other than the fiberwise capping.  There are chain maps $\Phi_{F_0}^{F_1}$, $\Phi_{F_0}^{H'}$, $\Phi_{H'}^{F_1}$, inducing the maps $\Psi$ of the  diagram (\ref{introdiag}).  Further, Corollary \ref{bigcor} implies that $\Phi_{F_0}^{F_1}\co CF_{2r}^{[a,b]}(F_0)\to CF_{2r}^{[a,b]}(F_1)$ maps $N_0$ to $N_1$, while we show in Proposition \ref{p2} that $\Phi_{F_0}^{F_1}(x_0)=x_1$.  Meanwhile, Proposition \ref{p1} shows that $x_0$ and $x_1$ are cycles in their respective complexes, while the spaces of degree-$2r$ boundaries in $CF_{2r}^{[a,b]}(F_i)$ are contained in $N_i$.  We then show in the proof of Proposition \ref{p3} that, if our result were false, one could construct a cycle $c'\in N_0$ with the property that $\Phi_{F_0}^{H'}(c')=\Phi_{F_0}^{H'}(x_0)$.  But then since $\Phi_{F_0}^{F_1}$ induces the same map on homology as $\Phi_{H'}^{F_1}\circ \Phi_{F_0}^{H'}$, $\Phi_{F_0}^{F_1}(x_0-c')$ would need to be a boundary, which cannot be the case since the foregoing implies that  $\Phi_{F_0}^{F_1}(x_0-c')\notin N_1$.  

The appendix is concerned with some facts about local Hamiltonian Morse-Bott Floer homology, which are probably at least similar to results known to experts, and which are needed in the proof of Proposition \ref{p1}.

\subsection*{Acknowledgements} I am grateful to B. G\"urel for helpful conversations.  Most of this work was carried out while I was at Princeton University.

\section{Filtered Floer homology in symplectic disc bundles}

As our input we take: \begin{itemize}\item a closed connected symplectic manifold $(M,\omega_0)$ of (real) dimension $2m$;
\item a Morse function $h\co M\to \mathbb{R}$ having just one local maximum (\emph{i.e.}, just one critical point of index $2m$; Theorem 8.1 of \cite{Mil} shows how to construct such an $h$). In the situation of case (ii) of Theorem \ref{main2} we will also assume that $-h$ has no critical points of index $1$, so that $h$ has no critical points of index $2m-1$;
\item a Hermitian vector bundle $\pi\co E\to M$ with (complex) rank $r$; and \item a unitary connection $A$ on $E$, which in particular provides a splitting $TE=T^{hor}E\oplus T^{vt}E$.  We assume that the connection $A$ is trivial on some neighborhood of the (finitely many) critical points of $h$.\end{itemize}

Where $\langle\cdot,\cdot\rangle$ denotes the Hermitian inner product on $E$, define \[ L\co E\to \mathbb{R}\mbox{ by }L(x)=\frac{1}{4}\langle x, x\rangle,\] and define \[ \theta\in \Omega^1(E) \mbox{ by }\theta_x(v)=-dL(i(v^{vt}))\mbox{ for }v\in T_xE, \] where $v^{vt}$ denotes the vertical part of $v\in T_xE$ as given by the connection $A$.  Routine computations then show that: \begin{itemize} \item[(i)] $d\theta\in \Omega^2(E)$ restricts to each $\mathbb{C}^r$ fiber of $E\to M$ as the standard symplectic form on $\mathbb{C}^r$ given by the imaginary part of the Hermitian metric;
\item[(ii)] If $u\in T^{hor}E$ and $v\in T^{vt}E$ then $d\theta(u,v)=0$; \item[(iii)] If $u,w\in T^{hor}_{x}E$, then \[ (d\theta)_x(u,w)=\langle -i(F_A)_{\pi(x)}(\pi_*u,\pi_*w)x,x\rangle,\] where $F_A\in \Omega^2(M,\mathfrak{u}(E))$ denotes the curvature $2$-form of the Hermitian connection $A$.\end{itemize}

For $R>0$, define \[ E(R)=\{x\in E| \langle x,x\rangle\leq R^2\},\] and, for $0<R_1<R_2$, \[ E(R_1;R_2)=\overline{E(R_2)\setminus E(R_1)}.\]  Let $J_0$ be an $\omega_0$-compatible almost complex structure on $M$; this determines a metric $g_0$ on $M$ by $g_0(v,w)=\omega_0(v,J_0 w)$, in particular allowing us to measure the $L^{\infty}$ norm $\|F_A\|_{\infty}$ of the $\mathfrak{u}(E)$-valued $2$-form $F_A$ on $M$.  (iii) above then implies that if $\|F_A\|_{\infty}\langle x,x\rangle<1$,  $\pi^*\omega_0+d\theta$ restricts nondegenerately to $T^{hor}_{x}E$.  Thus provided that $R^2\leq \frac{1}{2}\|F_A\|_{\infty}^{-1}$, $\omega:=\pi^{*}\omega_0+d\theta$ defines a symplectic form on $E(R)$, which tames the almost complex structure $\bar{J}$ on $E$ obtained by lifting $J_0$ to $T^{hor}E$ and using the complex vector bundle structure on $T^{vt}E$.  Choose an $R$ such that this is the case.  Below, unless otherwise noted, we always measure distances in $E(R)$ using the metric $g$ on $E(R)$ given by $g(v,w)=\frac{1}{2}(\omega(v,\bar{J}w)+\omega(w,\bar{J}v))$;  by decreasing $R$ if necessary, for convenience let us also assume that the $g$-distance between any two points in the same fiber of $E(R)$ is equal to their distance as considered within their common fiber using the Hermitian metric on the fiber.

We wish to construct a version of filtered Floer homology for Hamiltonians of a particular kind on $E(R)$.  Let $\widetilde{\mathcal{L}}$ be the space of equivalence classes $[\gamma,w]$ where $\gamma\co \mathbb{R}/\mathbb{Z}\to E(R)$ is a contractible loop; $w\co D^2\to E(R)$ satisfies (identifying $\mathbb{R}/\mathbb{Z}$ with $D^2$ by $t\mapsto e^{2\pi it}$) $w|_{\partial D^2}=\gamma$; and $[\gamma,w]$ is deemed equivalent to $[\gamma,w']$ if and only if the classes $[\omega], c_1(TE(R))\in H^2(E(R);\mathbb{Z})$ both evaluate trivially on the sphere obtained by gluing $w'$ to $w$ orientation-reversingly along their common boundary $\gamma$.  
  Given an arbitrary Hamiltonian $H\co (\mathbb{R}/\mathbb{Z})\times E(R)\to \mathbb{R}$, the action  functional $\mathcal{A}_H\co \widetilde{\mathcal{L}}\to\mathbb{R}$ is then defined by \[ \mathcal{A}_H([\gamma,w])=-\int_{D^2}w^*\omega-\int_{0}^{1}H(t,\gamma(t))dt.\]  The critical points of $\mathcal{A}_H$ are precisely those $[\gamma,w]$ where $\gamma$ is a contractible $1$-periodic orbit of the (time-dependent) Hamiltonian vector field $X_H$, where our sign convention is that $dH=\omega(X_H,\cdot)$.
  For real numbers $a<b$, the filtered Floer complex $CF^{[a,b]}_{*}(H)$ is the free $\mathbb{Z}_2$-module generated by those critical points $[\gamma,w]$ of $\mathcal{A}_H$ having $a\leq \mathcal{A}_H([\gamma,w])\leq b$ ($a$ and $b$ will always be taken finite in this paper, so no Novikov completion is needed).  The Floer boundary operator should enumerate solutions $u\co \mathbb{R}\times (\mathbb{R}/\mathbb{Z})\to E(R)$ to the equation \begin{equation}\label{floereqn} \frac{\partial u}{\partial s}+J_t(u(s,t))\left(\frac{\partial u}{\partial t}-X_H(t,u(s,t))\right)=0\end{equation} which connect two generators of $CF^{[a,b]}(H)$ ($J_t$ is a $t$-parametrized family of almost complex structures, which in practice will be close to the $\bar{J}$ of the previous paragraph).  Of course, since $E(R)$ is a manifold with boundary, compactness of the relevant solution spaces is not obvious, and indeed would not hold for many choices of $H$ and $[a,b]$.  However, compactness does hold 
 for certain choices of $H$ and $[a,b]$ which are suitable for our purposes, as we now set about proving.
\subsection{A radius-energy estimate}  
  Choose a number $\alpha\leq R/4$.  The Hamiltonians that we consider will have the form \begin{equation}\label{hform} H(t,x)=B+2\pi L(x)+f(t,x)+\delta h(\pi(x))\end{equation} where $B$ is a constant; $supp(f)\subset (\mathbb{R}/\mathbb{Z})\times E(\alpha)$;  $\delta$ is a small positive number, which in particular should satisfy $\|\delta h\circ \pi \|_{C^2}\leq \alpha/4$; and as before $L(x)=\frac{1}{4}\langle x,x\rangle$ and $h\co M\to\mathbb{R}$ is our fixed Morse function. 
Thus $H$ has a fairly specific form on $E(\alpha;R)=\overline{E(R)\setminus E(\alpha)}$, but the presence of the term $f$ means that we effectively assume nothing about the behavior of $H$ on the interior of $E(\alpha)$; indeed our intention is to apply our work to Hamiltonians having arbitrarily large derivatives in certain subsets of $E(\alpha)$.   

Note that, under the natural identification of $T_{x}^{vt}E$ with the fiber $E_{\pi(x)}$ of the Hermitian bundle $E\to X$ which contains $x$, the Hamiltonian vector field of $2\pi L$ is given by $X_{2\pi L}(x)=-\pi i x$.  So if $\{\phi_t\}_{t\in [0,1]}$ denotes the Hamiltonian flow of $2\pi L$, we have $\phi_t(x)=e^{-\pi i t}x$, and in particular $\phi_1(x)=-x$.  For $v\in TE$ denote $|v|^2=g(v,v)$.

\begin{lemma}\label{loopbig} If $H$ is as in (\ref{hform}) and if the $C^1$ path $\gamma\co [0,1]\to E(\alpha;R)$ satisfies $\gamma(0)=\gamma(1)$, then \[ \int_{0}^{1}|\dot{\gamma}(t)-X_H(\gamma(t))|^2 dt\geq \left(\frac{7\alpha}{4}\right)^2 \]
\end{lemma}

\begin{proof}  Where $\{\phi_t\}_{t\in [0,1]}$ denotes the Hamiltonian flow of $2\pi L$, let $\eta(t)=(\phi_{t})^{-1}(\gamma(t))$, \emph{i.e.}, $\eta(t)=e^{\pi i t}\gamma(t)$.  Then \[ \dot{\eta}(t)=\pi ie^{\pi i t}\gamma(t)+ e^{\pi i t}\dot{\gamma}(t)=e^{\pi i t}(-X_{2\pi L}(\gamma(t))+\dot{\gamma}(t)),\] and so \[ |\dot{\gamma}(t)-X_{2\pi L}|=|\dot{\eta}(t)|.\]  Now since by hypothesis $\gamma(0)=\gamma(1)\in E(\alpha;R)$ we have $\eta(0)=-\eta(1)$ and $\langle \eta(0),\eta(0)\rangle\geq \alpha^2$, so that $\eta$ is a path of length at least $2\alpha$.  So, using the Schwarz inequality, we get \begin{align*} &\left(\int_{0}^{1}|\dot{\gamma}(t)-X_H(\gamma(t))|^{2}dt\right)^{1/2}\geq \int_{0}^{1}|\dot{\gamma}(t)-X_H(\gamma(t))|dt\\&\geq \int_{0}^{1}|\dot{\gamma}(t)-X_{2\pi L}(\gamma(t))|-\int_{0}^{1}|X_{2\pi L}(\gamma(t))-X_H(\gamma(t))|dt\geq 7\alpha/4,\end{align*} since the hypothesis on $H$ and the assumption that $\gamma$ remains outside $E(\alpha)^{\circ}$ imply that $|X_{2\pi 
 L}(\gamma(t))-X_H(\gamma(t))|\leq \alpha/4$.

\end{proof}

\begin{lemma}\label{fast} If $H$ is as in (\ref{hform}) and if the $C^1$ path $\gamma\co [0,1]\to E(R)$ has image which intersects both $E(\alpha)$ and $E(2\alpha;R)$ then 
\[ \int_{0}^{1}|\dot{\gamma}(t)-X_H(\gamma(t))|^2 dt\geq \left(\frac{3\alpha}{4}\right)^2\]

\end{lemma}

\begin{proof} The hypothesis implies that we can choose $t_0,t_1\in [0,1]$ such that \linebreak[6]$\langle \gamma(t_0),\gamma(t_0)\rangle=\alpha^2$, $\langle \gamma(t_1),\gamma(t_1)\rangle=4\alpha^2$; for simplicity assume $t_0<t_1$.  If necessary by increasing $t_0$ and decreasing $t_1$ we may assume that $\langle\gamma(t),\gamma(t)\rangle\in [\alpha^2,4\alpha^2]$ for all $t\in [t_0,t_1]$.  Now the hypothesis on $H$ implies that (since $\gamma(t)\notin E(\alpha)^{\circ}$ for $t\in [t_0,t_1]$) $dL(X_H(\gamma(t)))=0$ for $t\in [t_0,t_1]$.  So (since $|dL|\leq \frac{1}{2}(2\alpha)=\alpha$ on $E(\alpha;2\alpha)$) \[ \frac{3\alpha^{2}}{4}=\int_{t_0}^{t_1}dL(\dot{\gamma}(t))dt=\int_{t_0}^{t_1}dL(\dot{\gamma}(t)-X_H(\gamma(t)))dt \leq 
\alpha\int_{t_0}^{t_1}|\dot{\gamma}(t)-X_H(\gamma(t))|dt.\]  So the Schwarz inequality gives 
\begin{align*} \int_{0}^{1}|\dot{\gamma}(t)-X_H(\gamma(t))|^2 dt&\geq \left(\int_{0}^{1}|\dot{\gamma}(t)-X_H(\gamma(t))|dt\right)^2\\&\geq \left(\int_{t_0}^{t_1}|\dot{\gamma}(t)-X_H(\gamma(t))|dt\right)^2\geq \left(\frac{3\alpha}{4}\right)^2
\end{align*}

\end{proof}

Now a general $\omega$-tame almost complex structure $J$ on $E(R)$ induces a metric $g_J$ by $g_J(v,w)=\frac{1}{2}(\omega(v,Jw)+\omega(w,Jv))$; write $|v|_{J}^{2}=g_J(v,v)$, so that in our earlier notation 
$|v|=|v|_{\bar{J}}$.  For a $\mathbb{R}/\mathbb{Z}$-parametrized path of $\omega$-tame almost complex structures $J_t$ denote \[ \|J_t\|=\sup_{t\in\mathbb{R}/\mathbb{Z}, v\in TE(R):|v|>0}\{\frac{|v|}{|v|_{J_t}},\frac{|v|_{J_t}}{|v|}\}.\] (We only work with almost complex structures such that this is finite).

\begin{lemma}\label{sik}  There are constants $C$ and $\alpha_0$, depending only on $J_t$ and the function $\delta h\co M\to\mathbb{R}$, with the following property.  Suppose $\alpha\leq \alpha_0$, that $S\subset [0,1]\times(\mathbb{R}/\mathbb{Z})$ is a connected submanifold with boundary and that $u\co S\to E(2\alpha;3\alpha)$ satisfies the Floer equation (\ref{floereqn}) with $H(t,x)=2\pi L(x)+\delta h(\pi(x))$.  Suppose also that, for some $\ep,\ep'\in [0,\alpha/6)$, $u(\partial S)\subset \partial E(2\alpha+\ep;3\alpha-\ep')$, and that $u(\partial S)$ intersects both boundary components of $E(2\alpha+\ep;3\alpha-\ep')$.  Then  \[ Area(S)+\int_{S}\left|\frac{\partial u}{\partial s}\right|^{2}_{J_t}dsdt\geq C\alpha^2.\]
\end{lemma}

\begin{proof} Let $\tilde{E}=[-1,2]\times(\mathbb{R}/\mathbb{Z})\times E(R)$, and define an almost complex structure $\tilde{J}$ on $\tilde{E}$ by, on each $[-1,2]\times \{t\}\times E(R)$, setting $\tilde{J}|_{TE(R)}=J_t$, $\tilde{J}\partial_s=\partial_t+X_{H(t,\cdot)}$, and $\tilde{J}\partial_t=-\partial_s-J_tX_{H(t,\cdot)}$.  $\tilde{J}$ is tamed by the symplectic form $\tilde{\omega}=ds\wedge dt-dt\wedge dH+\omega$. Take $\alpha_0$ small enough that for any $x\in [0,1]\times(\mathbb{R}/\mathbb{Z})\times E(R/2)\subset \tilde{E}$ the exponential map at $x$ for the Riemannian manifold $(\tilde{E},g_{\tilde{J}})$ is an embedding on the ball of radius $\alpha_0$ in $T_{x}\tilde{E}$.  That $u\co S\to E(2\alpha;3\alpha)$ satisfies the Floer equation (\ref{floereqn}) is equivalent to the statement that the map $\tilde{u}\co S\to \tilde{E}$ defined by $\tilde{u}(z)=(z,u(z))$ is $\tilde{J}$-holomorphic; further we have (for any subsurface $S'\subset S$) \[ \int_{S'} \tilde{u}^*\tilde{\omega}=\int_{S'} ds\wedge dt+\int_{S'} \left|\frac{\partial u}{\partial s}\right|_{J_t}^2 dsdt.  \]
But our hypothesis implies that there is $z_0\in S$ such that $\langle z_0,z_0\rangle^{1/2}=5\alpha/2$.  Choosing $\eta\in (\alpha/4,\alpha/3)$ so that $\eta$ is a regular value of $\phi_{z_0}\co z\mapsto dist(\tilde{u}(z),\tilde{u}(z_0)),$ $S'=\phi_{z_0}^{-1}(B_{\eta}(\tilde{u}(z_0)))$ is a submanifold of $S$ such that $z_0\in S'$ and $\tilde{u}(\partial S')\subset \partial B_{\eta}(\tilde{u}(z_0))$.  So Proposition 4.3.1(ii) of \cite{Si} gives a constant $C'$ such that $\int_{S'}\tilde{u}^{*}\tilde{\omega}\geq C'\eta^2\geq \frac{C'}{16}\alpha^2$, from which the lemma immediately follows.  
\end{proof}

\begin{theorem}\label{c0}  There are constants $D$ and $\alpha_0$, depending only on $J$ and the function $\delta h\co M\to\mathbb{R}$, with the following property.  Let $\alpha\leq\alpha_0$, and let $H$ be of form (\ref{hform}).  Suppose that $u\co\mathbb{R}\times(\mathbb{R}/\mathbb{Z})\to E(R)$ is a solution to (\ref{floereqn}) such that there are $\gamma_{\pm}\co \mathbb{R}/\mathbb{Z}\to E(\alpha)$ with $u(s,\cdot)\to \gamma_{\pm}$ uniformly as $s\to\pm\infty$.  Suppose also that \[ E(3\alpha;R)\cap u(\mathbb{R}\times(\mathbb{R}/\mathbb{Z}))\neq\varnothing.\]  Then \[\int_{-\infty}^{\infty}\int_{0}^{1}\left|\frac{\partial u}{\partial s}\right|_{J_t}^{2}dsdt\geq D\alpha^4.\]

\end{theorem}

\begin{proof} Let \[ \mathcal{Z}=\{s\in\mathbb{R}|u(\{s\}\times (\mathbb{R}/\mathbb{Z}))\cap E(2\alpha;3\alpha)^{\circ}\neq \varnothing\};\] obviously $\mathcal{Z}$ is an open subset of $\mathbb{R}$.  If $s\in \mathcal{Z}$, then either $\gamma(t)=u(s,t)$ satisfies the hypothesis of Lemma \ref{loopbig} or else $\gamma$ satisfies the hypothesis of Lemma \ref{fast}.  So (since (\ref{floereqn}) shows that $\left|\frac{\partial u}{\partial s}\right|_{J_t}=\left|\frac{\partial u}{\partial t}-X_H(u(s,t))\right|_{J_t}$) we have \[ \int_{-\infty}^{\infty}\int_{0}^{1}\left|\frac{\partial u}{\partial s}\right|_{J_t}^{2}dsdt\geq \|J_t\|^{-2}\int_{-\infty}^{\infty}\int_{0}^{1}\left|\frac{\partial u}{\partial t}-X_H\right|^{2}dsdt\geq \frac{9}{16}\|J_t\|^{-2}\alpha^2 m_{Leb}(\mathcal{Z}).\]  So (where $C$ is the constant of Lemma \ref{sik}) if $\mathcal{Z}\subset\mathbb{R}$ has $m_{Leb}(\mathcal{Z})\geq C\alpha^{2}/2$ then the theorem holds for $u$ provided that $D\leq 9C\|J_t\|^{-2}/16$.

There remains the case that $\mathcal{Z}$ has measure less than $C\alpha^{2}/2$ (by lowering $\alpha_0$ if necessary, assume that $C\alpha^2/2\leq 1$).   Define $\zeta\co \mathbb{R}\times(\mathbb{R}/\mathbb{Z})\to\mathbb{R}$ by $\zeta(s,t)=\langle u(s,t),u(s,t)\rangle^{1/2}$; our hypothesis implies that $[\alpha,3\alpha]\subset Im(\zeta)$.  Choose $\ep,\ep'\in (0,\alpha/6)$ such that $2\alpha+\ep$ and $3\alpha-\ep'$ are regular values of $\eta$.  $\zeta^{-1}([2\alpha+\ep,3\alpha-\ep'])$ is then a submanifold with boundary of $\mathbb{R}\times(\mathbb{R}/\mathbb{Z})$, at least one of whose connected components, say $S$, has the property that $\zeta(\partial S)=\{2\alpha+\ep,3\alpha-\ep'\}$.  (For instance, a path in $\mathbb{R}\times (\mathbb{R}/\mathbb{Z})$ from $\zeta^{-1}(\{2\alpha+\ep\})$ to $\zeta^{-1}(\{3\alpha-\ep'\})$ can easily be seen to have a segment from $\zeta^{-1}(\{2\alpha+\ep\})$ to $\zeta^{-1}(\{3\alpha-\ep'\})$ that is contained within $\zeta^{-1}([2\alpha+
 \ep,3\alpha-\ep'])$, and we can take for $S$ the connected component of 
$\zeta^{-1}([2\alpha+\ep,3\alpha-\ep'])$ which contains this segment.)  Now for each $s$ having the property that there is $t$ with $(s,t)\in S$, we have $u(s,t)\in E(2\alpha;3\alpha)$, and so $s\in\mathcal{Z}$. Thus $S\subset\mathcal{Z}\times (\mathbb{R}/\mathbb{Z})$.  Now $\mathcal{Z}$ is a disjoint union of open intervals and $m_{Leb}(\mathcal{Z})\leq C\alpha^{2}/2\leq 1$, so since $S$ connected there is a single open subinterval of $\mathcal{Z}$, say $I\subset\mathcal{Z}$, such that $S\subset I\times(\mathbb{R}/\mathbb{Z})$.  Precomposing $u$ with a translation in the $s$-variable if necessary, we may as well assume that $I\subset [0,1]$.
  Then Lemma \ref{sik} gives that \[   
\int_{S}\left|\frac{\partial u}{\partial s}\right|_{J_t}^{2}dsdt\geq C\alpha^2-Area(S).\]  
But since $S\subset I\times(\mathbb{R}/\mathbb{Z})\subset \mathcal{Z}\times (\mathbb{R}/\mathbb{Z})$ we have $Area(S)\leq C\alpha^2/2$, and so
\[ \int_{-\infty}^{\infty}\int_{0}^{1}\left|\frac{\partial u}{\partial s}\right|_{J_t}^2dsdt\geq 
\int_{S}\left|\frac{\partial u}{\partial s}\right|_{J_t}^2dsdt\geq C\alpha^{2}/2.\]  So (since $\alpha\leq R/4$)
the theorem holds for $u$ provided that $D\leq 8C/R^2$.

Thus taking $D=\min\{8C/R^2,9C\|J_t\|^{-2}/16\}$ completes the proof.
\end{proof}

\begin{remark} \label{dindep}In fact, inspection of the proof of Lemma \ref{sik} and of the result of \cite{Si}
cited therein shows that, if the $J_t$ are chosen in a suitably small $C^2$-neighborhood of the standard almost complex structure $\bar{J}$, and if the function $\delta h$ is chosen in a sufficiently small $C^2$-neighborhood of $0$, then the constants $\alpha_0$, $C$, and $D$ of Lemma \ref{sik} and Theorem \ref{c0} can be taken independent of the particular $J_t$ and $\delta h$ from within these neighborhoods.  (Alternately, at least if $D\alpha^4$ is less than the minimal energy of a $J$-holomorphic sphere in $E(R)$, this can be seen as a direct consequence of Gromov compactness).  Also, since the proof of Theorem \ref{c0} makes use only of the behavior of $u$ on $u^{-1}(E(\alpha;3\alpha))$, it is enough to assume that the restriction of $u$ to  $u^{-1}(E(\alpha;3\alpha))$ satisfies (\ref{floereqn}).
\end{remark}
\subsection{Floer homology}
With the above $C^0$-estimate established, the definition of our Floer groups becomes an application of standard machinery.  Let $\hbar$ be equal to one-half of the minimal energy of a nonconstant $\bar{J}$-holomorphic sphere in $E(R)$; Gromov compactness of course implies that this is a positive number and that, if $J_t$ is a $t$-parametrized family of almost complex structures which are sufficiently $C^2$-close to $\bar{J}$ then, for all $t$, any nonconstant $J_t$-holomorphic sphere will have energy at least $\hbar$.  
 
Choose any $\alpha>0$ with the property that $\alpha<\alpha_0$ and $D\alpha^4<\hbar$ (where, as in Remark \ref{dindep}, $\alpha_0$ and $D$ are chosen to satisfy Theorem \ref{c0} for any $J_t$ and $\delta h$ sufficiently $C^2$-close to $\bar{J}$ and $0$).  Let $H$ be a Hamiltonian of the form (\ref{hform}) and let $a<b$ be real numbers such that $b-a<D\alpha^4$. Note that the form we are assuming for $H$ implies that all $1$-periodic orbits $\gamma\co \mathbb{R}/\mathbb{Z}\to E(R)$ of the Hamiltonian vector field $X_H$ of $H$ are contained within the region $E(\alpha)$; assume furthermore that $H$ has the property that all of its one-periodic orbits $\gamma$ are nondegenerate in the sense that, where $\phi_H$ is the time-one map of $X_H$, the linearization $d\phi_H\co T_{\gamma(0)}E(R)\to T_{\gamma(0)}E(R)$ does not have $1$ as an eigenvalue.  Use the notation $[\gamma,w]$ to denote the equivalence class of a (nondegenerate) contractible one-periodic orbit $\gamma$ of $X_H$ together with a nullhomotopy $w\co D^2\to E(R)$ of $\gamma$, with $[\gamma,w]$ equivalent to $[\gamma,w']$ provided that both $[\omega]$ and $c_1$ vanish on the sphere obtained by gluing $w$ and $w'$ orientation-reversingly along their common boundary $\gamma$.  Since $\gamma$ is nondegenerate, any such object $[\gamma,w]$ has a well-defined Maslov index $\mu_{H}([\gamma,w])$; we adopt the conventions of  Section 2 of \cite{Sal} for the definition of this index.   For any integer $k$, define \[ CF^{[a,b]}_{k}(H)=\left\{\left.\sum_{i=1}^{l}a_i[\gamma_i,w_i]\right|l\in\mathbb{N}, a_i\in\mathbb{Z}_2, a\leq \mathcal{A}_H([\gamma_i,w_i])\leq b, \mu_{CZ}([\gamma_i,w_i])=k\right\}\] and \[ CF_{*}^{[a,b]}(H)=\bigoplus_{k\in\mathbb{Z}}CF^{[a,b]}_{k}(H).\]
 
In the usual way, one then defines the matrix elements of the Floer boundary operator $\partial_{H,J_t}^{[a,b]}$ by setting  $\langle \partial_{H,J_t}^{[a,b]}[\gamma_-,w_-],[\gamma_+,w_+]\rangle$ equal to zero when $\mu_{H}([\gamma_+,w_+])\neq \mu_{H}([\gamma_-,w_-])-1$, and otherwise equal to the number modulo two of solutions (modulo $s$-translation)  $u\co \mathbb{R}\times S^1\to E(R)$ to (\ref{floereqn}) for a generic path $J_t$ of almost complex structures $C^2$-close to $\bar{J}$, having the property that $u(s,\cdot)\to \gamma_{\pm}$ as $s\to \pm\infty$ and $[\gamma_+,w^-\#u]=[\gamma_+,w_+]$ where $\gamma_-\#u$ denotes the disc obtained by gluing the cylinder $u$ to the disc $w_-$ along their common boundary component $\gamma_-$.  Note that any such $u$ has \[ Energy(u):=\int_{-\infty}^{\infty}\int_{0}^{1}\left|\frac{\partial u}{\partial s}\right|_{J_t}^{2}dsdt=\mathcal{A}_{H}([\gamma_-,w_-])-\mathcal{A}_H([\gamma_+,w_+])< D\alpha^4<\hbar.\]  The fact that $Energy(u)<  \hbar$ precludes the bubbling off of holomorphic spheres in sequences of such $u$, while the fact that $Energy(u)< D\alpha^4$ implies, via Theorem \ref{c0}, that all such $u$ are \emph{a priori} contained in the region $E(3\alpha)$.  Gromov--Floer compactness and gluing then yield in the standard way that $(\partial_{H,J_t}^{[a,b]})^2=0$, and so we obtain ($[a,b]$-filtered) Floer homology groups \[  HF_{*}^{[a,b]}(H),\] which are independent of the choice of $J_t$, provided that $J_t$ is chosen from a certain set $\mathcal{J}^{reg}(H)$ having residual intersection with a small neighborhood of $\bar{J}$.  (When $[a,b]$ is understood from the context we will often write $\partial_{H,J_t}$ for $\partial_{H,J_t}^{[a,b]}$.)

We introduce the following standard definition:
\begin{definition} Let 
$H^-$ and $H^+$ be two nondegenerate Hamiltonians on $E(R)$ such that \begin{itemize} \item $H^-$ and $H^+$ have the form (\ref{hform}), \item $H^-\leq H^+$ everywhere, and  \item $(H^+-H^-)|_{E(\alpha;R)}$ is constant, \end{itemize}
and let $J_{t}^{\pm}$ be families of almost complex structures close to $\bar{J}$.  A \textbf{monotone \linebreak[6] homotopy} from $(H^-,J^{-}_{t})$ to $(H^+,J^+)$ is a path $(H^s,J_{s,t})$ ($s\in\mathbb{R}$) of pairs consisting of Hamiltonians $H^s$ and almost complex structures $J_{s,t}$ such that \begin{itemize} \item $\frac{\partial H^s}{\partial s}\geq 0$ everywhere;\item each $H^s$ has form (\ref{hform}) with the function $\delta h$ independent of $s$ (in particular, the vector fields $X_{H^s}$ all restrict in the same way to $E(\alpha;R)$);\item For some $S>0$ we have  $(H^s,J_{s,t})=(H^-,J^{-}_{t})$ for $s<-S$ while $(H^s,J_{s,t})=(H^+,J_{t}^{+})$ for $s>S$.  \end{itemize}\end{definition} 

Let $(H^s,J_{s,t})$ be a monotone homotopy from $(H^-,J^{-}_{t})$ to $(H^+,J^{+}_{t})$, with $J_{s,t}$ $C^2$-close to (the constant path at) $\bar{J}$, and consider solutions $u\co \mathbb{R}\times S^1\to E(R)$ to the equation \begin{equation}\label{htopyeqn}  \frac{\partial u}{\partial s}+J_{s,t}(u(s,t))\left(\frac{\partial u}{\partial t}-X_{H^s}(t,u(s,t))\right)=0 \end{equation}  having the property that $u(s,\cdot)\to \gamma_{\pm}$ and $[\gamma_+,w_-\#u]=[\gamma_+,w_+]$ for generators $[\gamma_{\pm},w_{\pm}]$ of $CF^{[a,b]}_{*}(H^{\pm})$.  One has the formula \begin{align}\label{htopyenergy} \mathcal{A}_{H^-}([\gamma_-,w_-])-\mathcal{A}_{H^+}([\gamma_+,w_+])&=\int_{-\infty}^{\infty}\int_{0}^{1}
\left|\frac{\partial u}{\partial s}\right|_{J_{s,t}}^{2}dsdt+\int_{-\infty}^{\infty}\int_{0}^{1}\frac{\partial H^s}{\partial s}(t,u(s,t))dsdt \nonumber\\&\geq Energy(u)\end{align} since we assume that $\frac{\partial H^s}{\partial s}\geq 0$. In particular, since $Energy(u)\geq 0$ we have $\mathcal{A}_{H^-}([\gamma_-,w_-])\geq \mathcal{A}_{H^+}([\gamma_+,w_+])$, while since $[\gamma_{\pm},w_{\pm}]$ are required to be generators of $CF^{[a,b]}_{*}(H^{\pm})$ and so have actions differing by at most $b-a<D\alpha^4<\hbar$, we have $Energy(u)<\min\{\hbar,D\alpha^4\}$.  So, just as with the definition of the Floer boundary operator, Gromov compactness together with the radius-energy estimate Theorem \ref{c0} establish compactness of the space of such $u$ (since, as in the last sentence of Remark \ref{dindep}, the fact that $X_{H^s}$ is independent of $s$ outside $E(\alpha)$ means that $u$ actually satisfies (\ref{floereqn}) on $u^{-1}(E(\alpha;3\alpha))$).  Provided that  $J_{t}^{\pm}\in\mathcal{J}^{reg}(H^{\pm})$ and that $J_{s,t}$ is chosen from a certain set $\tilde{\mathcal{J}}^{reg}(H^s)$ having residual intersection with a small neighborhood of the constant path at $\bar{J}$, this allows us to define a map $\Phi_{H^s,J_{s,t}}\co CF^{[a,b]}_{*}(H^-)\to CF^{[a,b]}_{*}(H^+)$ by counting those solutions to (\ref{htopyeqn}) which connect generators $[\gamma_{\pm},w_{\pm}]$ having equal Maslov index.  By definition, the monotone homotopy $(H^s,J_{s,t})$ will be called \textbf{regular} if $J_{t}^{\pm}\in\mathcal{J}^{reg}(H^{\pm})$ and $J_{s,t}\in \tilde{\mathcal{J}}^{reg}(H^s)$.  The usual arguments (dating back to Theorem 4 of \cite{F}) involving a homotopy of homotopies and the Floer gluing theorem show that, whenever $(H^s,J_{s,t})$ is a regular monotone homotopy from $(H^-,J_{t}^{-})$ to $(H^+,J_{t}^{+})$, \begin{itemize} \item $ \partial_{H^+,J_{t}^{+}}\circ \Phi_{H^s,J_{s,t}}=\Phi_{H^s,J_{s,t}}\circ \partial_{H^-,J_{t}^{-}}$; \item The induced map $\Psi_{H^-}^{H^+}\co HF^{[a,b]}_{*}(H^-)\to HF^{[a,b]}_{*}(H^+)$ is independent of the choices of $H^s$ and $J_{s,t}$ (in fact, any two choices of $(H^s,J_{s,t})$ induce chain homotopic maps $\Phi_{H^s,J_{s,t}}$); and \item If $H^-\leq H^0\leq H^+$ we have \[ \Psi_{H^0}^{H^+}\circ \Psi_{H^-}^{H^0}=\Psi_{H^-}^{H^+}.\]\end{itemize}

We remark at this point in the situation of greatest interest to us, the ``period map'' $[\omega]\co \pi_2(E(R))\to \mathbb{R}$ will have dense image, as a result of which for each periodic orbit $\gamma$ of $X_H$ there will be infinitely many different choices of $w$ for which $[\gamma,w]$ is a generator of $CF^{[a,b]}_{*}(H)$, and the actions of these generators will fill up a dense subset of $[a,b]$.  As one varies the Hamiltonian, even if the periodic orbits vary in a simple way, there likely will be generators $[\gamma,w]$ moving in and out of the ``action window'' $[a,b]$ rather frequently, as a result of which (in distinct contrast to the situation in \cite{GG2}) the maps $\Psi_{H^-}^{H^+}\co HF^{[a,b]}_{*}(H^-)\to HF^{[a,b]}_{*}(H^+)$ will typically be far from being either injective or surjective.    However, in suitable situations, it will still be possible to obtain useful information about $\Psi_{H^-}^{H^+}$.

The following simple proposition will be helpful to us in the proof of Proposition \ref{p3} below.

\begin{prop}\label{ipluslower}  Suppose that $(H^s,J_{s,t})$ is a regular monotone homotopy from $(H^-,J^{-}_{t})$ to $(H^+,J^{+}_{t})$, and suppose that $\gamma\co \mathbb{R}/\mathbb{Z}\to E(R)$ has the property that, for each $s$, $\gamma$ is a nondegenerate $1$-periodic orbit of $X_{H^s}$.  Suppose also that, for all $t\in\mathbb{R}/\mathbb{Z}$, we have $H^-(t,\gamma(t))=H^+(t,\gamma(t))$.  Then, for all $w\co D^2\to E(R)$ such that $\mathcal{A}_{H^-}([\gamma,w])\in [a,b]$, we have \[ \Phi_{H^{s},J_{s,t}}([\gamma,w])=[\gamma,w]+\sum c_{[\gamma',w']}[\gamma',w'] \] where \[ \mathcal{A}_{H^+}([\gamma',w'])<\mathcal{A}_{H^+}([\gamma,w])\mbox{ whenever }c_{[\gamma',w']}\neq 0.\]
\end{prop}

\begin{proof}  First we note that the hypothesis implies that $\mu_{H^-}([\gamma,w])=\mu_{H^+}([\gamma,w])$.  Indeed, where $\phi^{s,t}$ is the time-$t$ map of the Hamiltonian flow of $H^s$, and $\phi^{\pm,t}$ is the time-$t$ map of the Hamiltonian flow of $H^{\pm}$, one has $\mu_{H^{\pm}}([\gamma,w])=n-\mu_{CZ}(A^{\pm})$ where $A^{\pm}\co [0,1]\to Sp(2n)$ is the path of symplectic matrices obtained by setting $A^{\pm}(t)$ equal to the linearization at $\gamma(0)$ of the map $\phi^{\pm,t}$, as measured via a symplectic trivialization of $w^*TM$, and $\mu_{CZ}$ is the Conley--Zehnder index (see \cite{Sal}, Section  2.6).  But, where $A^s\co [0,1]\to Sp(2n)$ is defined similarly with $\phi^{\pm,t}$ replaced by $\phi^{s,t}$, the paths $A^s$ give a homotopy from $A^-$ to $A^+$; moreover the fact that $\gamma$ is nondegenerate for each $H^s$ shows that none of the matrices $A^s(1)$ has $1$ as an eigenvalue.  Hence the homotopy invariance of the Conley--Zehnder index (\cite{Sal}, Section 2.4) shows that $\mu_{CZ}(A^+)=\mu_{CZ}(A^-)$ and hence that  
$\mu_{H^-}([\gamma,w])=\mu_{H^+}([\gamma,w])$.

Given this, the fact that $\dot{\gamma}(t)=X_{H^s}(\gamma(t))$ for every $s$ implies that setting $u(s,t)=\gamma(t)$ gives an index-zero solution to (\ref{htopyeqn}), regardless of the choice of $J_{s,t}$.  Meanwhile, if $u$ is any  solution to (\ref{htopyeqn}) which is \emph{not} of the form $u(s,t)=\gamma(t)$, asymptotic say to $[\gamma',w']$ as $s\to \infty$ and to $[\gamma,w]$ as $s\to -\infty$, then we must have $\frac{\partial u}{\partial s}(s,t)\neq 0$ for some $(s,t)\in \mathbb{R}\times (\mathbb{R}/\mathbb{Z})$, and hence $\int_{\mathbb{R}\times (\mathbb{R}/\mathbb{Z})}\left|\frac{\partial u}{\partial s}\right|^2dsdt>0$.  So since $\frac{\partial H^s}{\partial s}\geq 0$ everywhere, (\ref{htopyenergy}) implies that \[ \mathcal{A}_{H^-}([\gamma,w])>\mathcal{A}_{H^+}([\gamma',w']).\]  So since the hypothesis $H^-(t,\gamma(t))=H^+(t,\gamma(t))$ implies that $\mathcal{A}_{H^-}([\gamma,w])=\mathcal{A}_{H^+}([\gamma,w])$, the proposition follows directly from the definition of the map $\Phi_{H^{s},J_{s,t}}$. (No considerations of sign are needed since we are working modulo two.)
\end{proof}

\section{Special features of the Floer complexes of certain Hamiltonians}

The proof of Theorem \ref{main2} requires us to understand certain properties of the Floer complexes of Hamiltonians on $E(R)$ having a particular form. 

As a first step, we prove the following elementary fact.
\begin{prop} \label{strongyorke}Let $V$ be a $C^1$ vector field on a closed Riemannian manifold $(M,g_0)$, $0\leq \rho<1$, and $U\subset M$ an open neighborhood of the zero locus $V^{-1}(\{0\})$.  Then there is $\delta_0>0$ having the following significance.  If $0<\delta<\delta_0$, and if $x\co \mathbb{R}/\mathbb{Z}\to M$ satisfies \begin{itemize} \item $\dot{x}(t)=\delta V(x(t)) \mbox{ if }x(t)\in U$ and  \item $|\dot{x}(t)-\delta V(x(t))|\leq \rho\delta |V(x(t))| $ for all $t$,\end{itemize} then $x$ is the constant loop at some zero of $V$.
\end{prop}

\begin{remark} \label{yorke} When $\rho=0$, the Yorke estimate \cite{Y} shows that $\delta_0$ can be taken equal to $2\pi/C$ where $C$ is the Lipschitz constant of $V$ as measured via an embedding of $M$  in $\mathbb{R}^N$.  
\end{remark}

\begin{proof}  The hypothesis implies that $dist(x(s),x(t))\leq \delta(1+\rho)\|V\|_{C^0}$ for all $s,t\in\mathbb{R}/\mathbb{Z}$.  Write \[ d_U=\inf\{dist(p,q)|p\in V^{-1}(\{0\}),q\in M\setminus U\}\] 
and $\beta_0=\frac{1}{4}d_U\|V\|_{C^0}^{-1}$; the set \[ U'=\{p\in M|dist(p,M\setminus U)>d_U/2\}\] is then an open neighborhood of $V^{-1}(\{0\})$ with the property that, as long as $\delta_0\leq\beta_0$, if for some $t_0$ we have $x(t_0)\in U'$, then for all $t$ we have $x(t)\in U$ (and so $\dot{x}(t)=V(x(t))$ by the hypothesis on $x$).  Thus if $x(\mathbb{R}/\mathbb{Z})\cap U'\neq\varnothing$, then $x$ actually satisfies the hypothesis of the proposition with $\rho=0$, so that (using any $\delta_0\leq \min\{\beta_0,2\pi/C\}$) the proposition holds for $x$ by the Yorke estimate mentioned in Remark \ref{yorke}. 

As such, it suffices to consider those $x\co \mathbb{R}/\mathbb{Z}\to M$ whose images are contained entirely in $M\setminus U'$.  $M\setminus U'$ is a compact set on which $V$ vanishes nowhere, so let $\ep=\min_{p\in M\setminus U'}|V(p)|$.  
 Let $\gamma_0=(2\|V\|_{C^0})^{-1}injrad(M,g_0)$.  Where $W\subset \mathbb{R}^{\dim M}$ is a neighborhood of the origin containing a ball of radius $injrad(M,g_0)$, let $\phi\co W\to M$ be a normal coordinate chart with $\phi(\vec{0})=x(0)$.  Let $B_0=\phi(W)$.  Let $V_0$ denote the vector field on $B_0$ obtained by the parallel transport of $V(x(0))$ along geodesics in $B_0$ issuing from $x(0)$.  Extend $(\phi^{-1}_{*})_{x(0)}V(x(0))\in T_{\vec{0}}W$ to a constant vector field $\bar{V}_1$ on $W\subset \mathbb{R}^{\dim M}$, and  define a vector field $V_1$ on $B_0$ by $V_1(p)=(\phi_*)_{\phi^{-1}(p)}\bar{V}_1(\phi^{-1}(p))$.  
 
 Now  if $\delta<\gamma_0$, the image of $x$ is then contained inside $B_0$.   For $p\in B_0$, we then have \[ |V_0(p)-V(p)|\leq \|V\|_{C^1}dist(p,x(0))\] and \[ |V_0(p)-V_1(p)|\leq \zeta\|V\|_{C^0}dist(p,x(0)),\] where $\zeta>0$ is some number depending only on the metric $g_0$ (in particular, $\zeta$ can be taken independent of $x(0)$).  So, for each $t$, we have (by the hypothesis on $x$) \begin{align*} g_0(\dot{x}(t),&V_1(x(t)))\geq g_0(\dot{x}(t),V(x(t)))-|\dot{x}(t)||V_1(x(t))-V(x(t))|
\\&\geq (1-\rho)\delta |V(x(t))|^2-(1+\rho)\delta|V(x(t))|(1+\zeta)\|V\|_{C^1}dist(x(t),x(0))\\&\geq (1-\rho)\delta\ep^2-(1+\zeta)(1+\rho)^2\delta^2\|V\|_{C^1}^{2}>0\end{align*} provided that $\delta$ is less than some constant $\delta_0$ which depends only on $\zeta,\gamma_0,\rho,\ep,\|V\|_{C^1}$.  $\phi^{-1}(x(t))$ $(0\leq t\leq 1)$ is thus a path contained entirely within a neighborhood $W$ of the origin in $\mathbb{R}^{\dim M}$ whose velocity vector has strictly positive inner product with the nonzero \emph{constant} vector field $\bar{V}_1$ on $W$, and this precludes the possibility that $x(1)=x(0)$, contrary to the hypothesis of the theorem.

This shows that in fact if $\delta_0$ is small enough every $x$ satisfying the hypotheses of the theorem must pass through $U'$, and so must be a constant loop at a zero of $V$ by our earlier remarks.  
\end{proof}

Recall that at the outset we have fixed a Morse function $h\co M\to\mathbb{R}$, and chosen our connection $A$ to be trivial in a neighborhood of the set of critical points of $h$.
\begin{definition} For $\delta>0$, a \emph{$\delta h$-Hamiltonian} on $E(R)$ is an autonomous Hamiltonian $H\co E(R)\to\mathbb{R}$ of the form \[ H=f\circ L+\delta h\circ\pi,\] where $f\co \mathbb{R}\to\mathbb{R}$ is a smooth function (and again $L(x)=\frac{1}{4}\langle x,x\rangle$).
\end{definition}

Let us consider the Hamiltionian vector fields $X_H$ of $\delta h$-Hamiltonians $H$. Of course $X_H=X_{f\circ L}+\delta X_{h\circ \pi}$.  We have noted earlier that (under the canonical identification of $E(R)$ with the $R$-disc bundle in the restriction of $T^{vt}E$ to the zero section of $E$) $X_{2\pi L}(x)=-\pi i x$, so that \[ X_{f\circ L}(x)=-\frac{if'(L(x))}{2}x.\]  In particular $X_{f\circ L}(x)\in T^{vt}_{x}E(R)$, and $dL(X_{f\circ L})=0$.  Now we consider $X_{h\circ \pi}$.  Now \[ d(h\circ\pi)_x(v)=(dh)_{\pi(x)}(\pi_*v)=0\mbox{ for }v\in T^{vt}E(R),\] so since $T^{hor}E(R)$ is the orthogonal complement of $T^{vt}E(R)$ with respect to our symplectic form $\omega=\pi^{*}\omega_0+d\theta$, we have $X_{h\circ \pi}(x)\in T^{hor}_{x}E(R)$ for all $x$.  For $w\in T_{\pi(x)}M$  let $w^{\#}$ denote its horizontal lift to $T^{hor}_{x}E(R)$.  Let $Y_h$ denote the Hamiltonian vector field of $h\co M\to\mathbb{R}$ (using the symplectic form $\omega_0$ on $M$).   																

We find, for $w\in T^{hor}_{x}E$, \[ \omega_x(Y_{h}^{\#},w)=(\omega_0)_{\pi(x)}(Y_h,\pi_*w)+d\theta(Y_{h}^{\#},w)=d(h\circ\pi)_{x}(w)+d\theta(Y_{h}^{\#},w);\] thus \[ \iota_{Y_{h}^{\#}-X_{h\circ\pi}}\omega=\iota_{Y_{h}^{\#}}d\theta,\] and so \begin{equation}\label{hamdiff} |X_{h\circ\pi}(x)-Y_{h}^{\#}(x)|\leq \|F_A(\pi(x))\|\langle x,x\rangle |Y_{h}^{\#}(x)|.\end{equation}

This yields:
\begin{prop} Provided that $R^2<(2\|F_A\|_{\infty})^{-1}$, there is $\delta_0>0$ such that, if $\delta<\delta_0$, if $H$ is any $\delta h$-Hamiltonian, and if $x\co \mathbb{R}/\mathbb{Z}\to E(R)$ is any solution to $\dot{x}(t)=X_H(x(t))$, then $\pi\circ x\co \mathbb{R}/\mathbb{Z}\to M$ is the constant loop at some critical point of $h$. 
\end{prop}

\begin{proof}  First note that as long as $R^2<(2\|F_A\|_{\infty})^{-1}$, for each $w\in T^{hor}E(\alpha_0)$ we have $|d\theta(w,\bar{J}w)|\leq \frac{1}{2}\omega_0(\pi_*w,J_0\pi_*w)$, and so $\frac{1}{2}|\pi_*w|^{2}\leq |w|^2\leq \frac{3}{2}|\pi_*w|^2$.  So (\ref{hamdiff}) implies that, where $y(t)=\pi(x(t))$, \begin{align*} |\dot{y}(t)-\delta Y_h(y(t))|&\leq \sqrt{2}\|F_A(y(t))\|\langle x(t),x(t)\rangle|\delta Y_{h}^{\#}(x(t))|\\&\leq \sqrt{3} \|F_A(y(t))\|\langle x(t),x(t)\rangle \delta|Y_h(y(t))|.\end{align*}  Now the zeros of $Y_h$ are just the critical points of $h$; recall that we assumed the connection $A$ to be trivial (and hence to have vanishing curvature) on a neighborhood of each of these points.  This, together with the  hypothesis on $R$ (and the fact that $\langle x(t),x(t)\rangle\leq R^2$), show  that the above coefficient $\sqrt{3} \|F_A(y(t))\|\langle x(t),x(t)\rangle$ is at most $\sqrt{3}/{2}$, and vanishes near the zeros of $Y_h$.  Hence we may apply Proposition \ref{strongyorke} with $\rho=\sqrt{3}/2$ to deduce the result.
\end{proof}
\subsection{The perturbations and their periodic orbits}
Suppose that $H=f\circ L+\delta h\circ \pi$ is a $\delta h$-Hamiltonian, with $\delta<\delta_0$.  We have just established that the $1$-periodic orbits of $X_H$ all lie in the fibers $E_p$ for $p$ some critical point of $h\co M\to\mathbb{R}$.  To specifically identify these orbits, note that our calculations have shown that, for $x\in E_p$, we have $X_H(x)=-\frac{if'(L(x))}{2}x$, so that the time-one map $\phi_H=\phi_1$ of $X_H$ restricts to the fibers $E_p$ over critical points $p$ of $h$ as \[ \phi_H(x)=e^{-if'(L(x))/2}x \quad (x\in E_p);\] thus where \[ \mathcal{L}_f=\{0\}\cup \{L> 0|f'(L)\in 4\pi\mathbb{Z}\},\] the fixed points of $\phi_H$ are precisely the points lying on a sphere of radius $\ell$ in the fiber over a critical point of $h$, where $\ell\in\mathcal{L}_f$.  In particular (unless $\mathcal{L}_f=\{0\}$) $H$ will be a degenerate Hamiltonian, so if we wish to take its Floer homology we will need to perturb it.

We assume, as will be the case in our application, that $0<f'(0)<4\pi$ and that, for each $\ell\in \mathcal{L}_f\setminus\{0\}$ we have $f''(\ell)\neq 0$, which in particular implies that $\mathcal{L}_f$ is a discrete set.   Let $\beta_f\co [0,R]\to [0,1]$ be a smooth function with the property that $\beta_f(s)=0$ for $s$ in some neighborhood of $\{0,R\}$, while $\beta_f(s)=1$ for $s$ in some neighborhood $V_f$ of $\mathcal{L}_f\setminus \{0\}$.  (If, as in the application, we have $f'(L)=2\pi$ for $L\geq \alpha$, then the support of $\beta_f$ should be contained in $[0,\alpha]$.)  Let $U$ be a neighborhood of $Crit(h)$ over which the bundle $E|_U\to U$ and the connection $A$ are trivial (and choose a trivialization $E|_U\cong U\times\mathbb{C}^r$ in terms of which the connection is standard), let $V$ be a neighborhood of $Crit(h)$ compactly contained in $U$, and let $\chi\co M\to [0,1]$ be a smooth function such that $\chi|_V=1$ and $\chi|_{M\setminus U}=0$.  Let $(x_1+iy_1,\ldots,x_r+iy_r)$ be fiberwise complex coordinates for $E|_U$, and define $g_f\co E(R)\to\mathbb{R}$ by \[ g_f(u,x_1+iy_1,\ldots,x_r+iy_r)=\chi(u)\beta_f(\sum(x_{j}^{2}+y_{j}^{2})/4)y_1 \] for $(u,x_1+iy_1,\ldots,x_r+iy_r)\in E|_U\cong U\times \mathbb{C}^r$ and $g_f|_{E(R)\setminus E|_U}=0$.  
Let $\psi_{f,\ep}\co E(R)\to E(R)$ denote 
the time-$1$ map of the Hamiltonian flow of $\ep g_f$.
  Finally, define the Hamiltonian $H^{\ep}\co (\mathbb{R}/\mathbb{Z})\times E(R)\to\mathbb{R}$ by \[ H^{\ep}(t,x)=H(x)+\ep g_f(\phi_{t}^{-1}(x)),\] where $\phi_t$ denotes the time $t$-flow of $X_H$.  Thus (as a standard calculation shows) $H^{\ep}$ has time-1 map $\phi_{H^{\ep}}$ equal   to $\phi_H\circ \psi_{f,\ep}$.  Of course, as $\ep\to 0$, $\psi_{f,\ep}$ converges to the identity in any $C^k$-norm, and so any fixed points of $\phi_{H^{\ep}}$ must, for sufficiently small $\ep$, be close to fixed points of $\phi_H$.  Since $\beta_f$ vanishes near the $0$-section of $E(R)$, the fixed points of $\phi_{H^{\ep}}$ near the zero section coincide with those of $\phi_H$, and thus are precisely the critical points of $h\co M\to\mathbb{R}$.  Now we consider the fixed points of $\phi_{H^{\ep}}$ away from the zero section.  For small $\ep$, all of these fixed points  of $\phi_{H^{\ep}}$ must be contained in the interior of a region where $\chi\circ \pi=\beta_f\circ L=1$.   Now for any $z=(u,x_1+iy_1,\ldots,x_r+iy_r)$ in this region, we have \[\psi_{\ep,f}(u,x_1+iy_1,\ldots,x_r+iy_r)=(u,(x_1+\ep)+iy_1,\ldots,x_r+iy_r).\]  In particular, for sufficiently small $\ep$, if $z\in Fix(\phi^{H^{\ep}})$ then $\pi(\psi_{\ep,f}(z))=\pi(z)$.    Now since the connection is trivial over $V\subset M$, we have $\pi(\phi_H(z))=\eta_{\delta h}(\pi(z))$ for $z\in \pi^{-1}(V)$, where $\eta_{\delta h}\co M\to M$ is the time-one map induced by the Hamiltonian $\delta h$ on $(M,\omega_0)$.  So since (for $\delta<\delta_0$) the only fixed points of $\eta_{\delta h}$ are the critical points of $h$, it follows that if $z$ is a fixed point of $\phi_{H^{\ep}}=\phi_H\circ \psi_{\ep,f}$ (and $\ep$ is sufficiently small), it must be that $\pi(z)\in Crit(h)$.  So any such fixed point has the form $z=(p_j,x_1+iy_1,\ldots,x_r+iy_r)$ where $p_j\in Crit(h)$; the condition for $z$ to be a fixed point is that \begin{equation}\label{fixed1} \phi_H(p_j, (x_1+\ep)+iy_1,\ldots,x_r+iy_r)=(p_j,x_1+iy_1,\ldots,x_r+iy_r).\end{equation}  Now one has $L\circ \phi_H=L$, so this forces $(x_1+\ep)^2=x_{1}^{2}$, \emph{i.e.}, $x_1=-\ep/2$.  So (\ref{fixed1}) reduces to \[ e^{-if'(L(z))/2}\left(\frac{\ep}{2}+iy_1,x_2+iy_2,\ldots,x_r+iy_r\right)=\left(-\frac{\ep}{2}+iy_1,x_2+iy_2,\ldots,x_r+iy_r\right).\]  Consideration of the first coordinate shows that $f'(L(z))/2$ must not be a multiple of $2\pi$, in view of which this forces $x_2=y_2=\ldots=x_r=y_r=0$, and \begin{equation}\label{fixed2} e^{-\frac{i}{2}f'\left(\frac{\ep^2+4y_{1}^{2}}{16}\right)}\left(\frac{\ep}{2}+iy_1\right)=-\frac{\ep}{2}+iy_1.\end{equation}
Let $\ell\in \mathcal{L}_f\setminus\{0\}$ (so $f'(\ell)=4\pi k$ for some integer $k$); we know \emph{a priori} that any fixed point $z$ of $H^{\ep}$ must (assuming $\ep$ is small enough) have $L(z)$ close to one such $\ell$.  Since we assume that $f''(\ell)\neq 0$ for all $\ell\in \mathcal{L}_f\setminus\{0\}$, for sufficiently small $\ep$ (\ref{fixed2}) has precisely two solutions $y_{+}^{\ell}$ and $y_{-}^{\ell}$ with $L$ close to $\ell$, of which the former has $y_1<0$ and $f'(\frac{\ep^2+4y_{1}^{2}}{16})$ slightly larger than $4\pi k$, while the latter has $y_1>0$ and $f'(\frac{\ep^2+4y_{1}^{2}}{16})$ slightly smaller than $4\pi k$.  Using that we also assume that $0<f'(0)<4\pi$, it is not difficult to see that for $\ep>0$ all fixed points of $\phi_{H^{\ep}}$ are nondegenerate.  This establishes:
\begin{prop}\label{orbits} Let $0<\delta<\delta_0$, and let $H=f\circ L+\delta h\circ \pi$ be a $\delta h$-Hamiltonian with the property that $f''(L)\neq 0$ whenever $f'(L)\in 4\pi\mathbb{Z}$ and $0<f'(0)<4\pi$.  Let $Crit(h)=\{p_1,\ldots,p_b\}$.  Then there is $\ep_0>0$ with the following property.  If $0<\ep<\ep_0$, then where $\phi_{H^{\ep}}$  is the time-one map  of the Hamiltonian $H^{\ep}$, the fixed points of $\phi_{H^{\ep}}$ are all nondegenerate, and are given by:
\begin{itemize}
\item $v_{1,0},\ldots,v_{b,0}$, the points on the zero section of $E(R)$ corresponding to the critical points $p_j$ of $h\co M\to\mathbb{R}$; and \item For each $\ell\in \mathcal{L}_f\setminus \{0\}$, points $v_{1,\ell}^{+},\ldots,v_{b,\ell}^{+}$, and   $v_{1,\ell}^{-},\ldots,v_{b,\ell}^{-}$ lying slightly above or slightly below the sphere of radius $\ell$ in the fibers of $E(R)\to M$ over $p_1,\ldots,p_b$.
\end{itemize}
\end{prop}

In particular, where $\phi_{H^{\ep}}^{t}$ denotes the time-$t$ flow of $H^{\ep}$, each of the $1$-periodic orbits $\gamma_{j,0}(t)=\phi_{H^{\ep}}^{t}(v_{j,0})$ or $\gamma_{j,\ell}^{\pm}(t)=\phi_{H^{\ep}}^{t}(v_{j,\ell}^{\pm})$ lies entirely in one of the fibers of $E(R)\to M$ over a critical point of $h$.  Recall that the generators of the Floer complex of $H^{\ep}$ are equivalence classes $[\gamma_{j,0},w]$ or $[\gamma_{j,l}^{\pm},w]$ where $w\co D^2\to E(R)$ is a nullhomotopy (or a ``capping'') of the orbit $\gamma_{j,0}$ or $\gamma_{j,l}^{\pm}$.  The fact that the orbits are all contained in single fibers  allows us to define, for each orbit $\gamma$, a ``fiberwise capping'' $w_0\co D^2\to E(R)$ by $w(se^{2\pi it})=s\gamma(t)$  (in particular if $\gamma$ is the constant orbit at a critical point of $h$ on the zero section then this is just the constant map at the zero section).  A general generator for the Floer complex then has the form $[\gamma_{j,0},w_0\#A]$ or $[\gamma_{j,\ell}^{\pm},w_0\#A]$, where \[ A\in \frac{\pi_2(E(R))}{\ker\langle [\omega],\cdot\rangle\cap \ker\langle c_1(TE(R)),\cdot\rangle}.\]  The actions and Maslov indices are related by \[ \mathcal{A}_{H^{\ep}}([\gamma,w_0\#A])=\mathcal{A}_{H^{\ep}}([\gamma,w_0])-\int_A\omega,\]\[ \mu_{H^{\ep}}([\gamma,w_0\#A])=\mu_{H^{\ep}}([\gamma,w_0])-2\langle c_1(TE(R)),A\rangle \] (see Section 2 of \cite{Sal} for the conventions we use on the Maslov index; in particular, if $\gamma$ is the constant  orbit at a critical point $p$ of a Morse function $G$ with its trivial capping, in our convention its Maslov index $\mu_G$ is equal to its Morse index as a critical point of $-G$).  Thus to understand the actions and gradings of the generators  of the Floer complex it is enough to understand the actions and gradings of the generators $[\gamma,w_0]$ where $w_0$ is the fiberwise capping of $\gamma$ (where $\gamma$ ranges among the $\gamma_{j,0}$ and $\gamma_{j,\ell}^{\pm}$).  A routine calculation using the characterization  of the Maslov and Conley--Zehnder indices from Sections 2.4 and 2.6 of \cite{Sal} gives the following result for the Maslov indices of the $[\gamma,w_0]$ (we leave this calculation to the reader; see Section 5.2.5 of \cite{GG1} for a sketch of a similar calculation, but note that we use different conventions both for Hamiltonian vector fields and for the normalization of the Maslov index):
\begin{prop}\label{grading} In the notation of Proposition \ref{orbits},  for $\ell\in \mathcal{L}_f\setminus \{0\}$ write \[ k(\ell)=\frac{f'(\ell)}{4\pi}.\]  Also let $n=m+r=\frac{1}{2}\dim E(R)$.  Then the Maslov indices for the periodic orbits of $H^{\ep}$ with their fiberwise cappings $w_0$ are, for $\ep$ sufficiently small, given by \begin{itemize} \item $\mu_{H^{\ep}}([\gamma_{j,0},w_0])=2n-ind_{h}p_j$;\item For $\ell\in \mathcal{L}_f\setminus \{0\}$, \[ \mu_{H^{\ep}}([\gamma_{j,\ell}^{+},w_0])=\left\{\begin{array}{ll} 2n-ind_{p_j}h+2rk(l) & f''(\ell)>0
\\ 2n-ind_{p_j}h+2rk(l)-1 & f''(\ell)<0\end{array}\right.\] and\[ \mu_{H^{\ep}}([\gamma_{j,\ell}^{-},w_0])=\left\{\begin{array}{ll} 2n-ind_{p_j}h+2r(k(l)-1)+1 & f''(\ell)>0
\\ 2n-ind_{p_j}h+2r(k(l)-1) & f''(\ell)<0\end{array}\right.\]\end{itemize}
\end{prop}
(In particular, for each $\ell\in \mathcal{L}_f\setminus\{0\}$ the orbits $[\gamma_{j,\ell}^{\pm},w_0]$ have Maslov indices differing from each other by $2r-1$, as would be expected since they are the two orbits that remain from an $S^{2r-1}$-family of periodic orbits of $H$ after perturbing $H$ to $H^{\ep}$).
\subsection{Restrictions on Floer trajectories}
We will be needing some information about the Floer complexes $CF^{[a,b]}(H^{\ep})$ of perturbations $H^{\ep}$ of particular $\delta h$-Hamiltonians $H=f\circ L+\delta h\circ \pi$; in our application the length of the interval $[a,b]$ will be rather small.  Lemma \ref{toprest} below will be a considerable help in this direction; that result, in turn, will depend on the following:

\begin{lemma}\label{iscpct} Given $\beta>0$, there are  constants $e_0$ and $\eta_0$ with the following property.  Let $w\co (-1,2)\times (\mathbb{R}/\mathbb{Z})\to M$ be a  map 
such that \[ 
\quad\left\|\frac{\partial w}{\partial s}+J_0\frac{\partial w}{\partial t}\right\|_{C^0}<\eta_0,\] and \[ \int_{(-1,2)\times(\mathbb{R}/\mathbb{Z})}\left|\frac{\partial w}{\partial s}\right|_{J_{0}}^{2}dsdt<e_0.\]  Then the diameter of $w([0,1]\times (\mathbb{R}/\mathbb{Z}))$ is no larger than $\beta$.
\end{lemma}

\begin{proof}  Standard estimates (\emph{e.g.}, Lemma 4.3.1 of \cite{MS}) show that there are $\delta,C>0$ such that if $u\co (-1,2)\times\mathbb{R}/\mathbb{Z}\to M$ is a $J_0$-holomorphic map with the property that $\int_{(-1,2)\times\mathbb{R}/\mathbb{Z}}|du|^{2}<\delta$ then $\|du\|_{L^{\infty}([0,1]\times \mathbb{R}/\mathbb{Z})}\leq C \left(\int_{(-1,2)\times\mathbb{R}/\mathbb{Z}}|du|^{2}\right)^{1/2}$.  In view of this, we can choose a parameter $e'_0>0$ with the property that if $u$ is $J_0$-holomorphic and $\int_{(-1,2)\times\mathbb{R}/\mathbb{Z}}|du|^{2}\leq e'_0$ then $u([0,1]\times\mathbb{R}/\mathbb{Z})$ has diameter at most $\beta/2$; our parameter $e_0$ will be equal to the minimum of $e'_0$ and half the minimal energy of a nonconstant $J_0$-holomorphic sphere.

Given this choice of $e_0$, we now prove the result.  If the result were false, there would be a sequence $w_n\co (-1,2)\times \mathbb{R}/\mathbb{Z}\to M$ 
with 
$\left\|\frac{\partial w_n}{\partial s}+J_0\frac{\partial w_n}{\partial t}\right\|_{C^0}\to 0$ such that $\int_{(-1,2)\times(\mathbb{R}/\mathbb{Z})}\left|\frac{\partial w_n}{\partial s}\right|_{J_0}^{2}dsdt<e_0$ and each $w_n([0,1]\times (\mathbb{R}/\mathbb{Z}))$ has diameter larger than $\beta$.  Let $\alpha_{n,s,t}$ be a sequence of $(-1,2)\times(\mathbb{R}/\mathbb{Z})$-parametrized  vector fields on $M$, varying continuously in $(s,t)\in (-1,2)\times(\mathbb{R}/\mathbb{Z})$, with the property that $\alpha_{n,s,t}(w_n(s,t))=\left(\frac{\partial w_n}{\partial s}+J_0\frac{\partial w_n}{\partial t}\right)(s,t)$; a straightforward patching argument shows that 
 we can arrange that $\|\alpha_{n,s,t}\|_{C^0}\to 0$ as $n\to\infty$, uniformly in $s$ and $t$.  

Define almost complex structures $\tilde{J}_n$ on $(-1,2)\times(\mathbb{R}/\mathbb{Z})\times M$ by setting $\tilde{J}_n\partial_s=\partial_t-J_0\alpha_{n,s,t}$, $\tilde{J}_n\partial_t=-\partial_s-\alpha_{n,s,t}$, and $\tilde{J}_n|_{T(\{(s,t)\}\times M)}=J_0$.  Define \[ W_n\co (-1,2)\times(\mathbb{R}/\mathbb{Z})\to (-1,2)\times(\mathbb{R}/\mathbb{Z})\times M\mbox{ by }W_n(s,t)=(s,t,w_n(s,t)).\]  The almost complex structures $\tilde{J}_n$ have been constructed so as to ensure that the $W_n$ are $\tilde{J}_n$-holomorphic maps.  Now where $i$ is the standard almost complex structure on $(-1,2)\times (\mathbb{R}/\mathbb{Z})$, the $\tilde{J}_n$ converge in $C^0$-norm to the product almost complex structure $i\times J_0$.  Hence, by Theorem 1 of \cite{IS}, after passing to a subsequence the $W_n$ converge, at least modulo bubbling, to a $i\times J_0$-holomorphic curve. 
Now the fact that the $W_n$ have form $(s,t)\mapsto (s,t,w_n(s,t))$ implies that the limiting bubble tree has a principal component of form $(s,t)\mapsto (s,t,w_{\infty}(s,t))$ (having energy, as measured by the symplectic form $ds\wedge dt+\omega_0$ and the almost complex structure $i\times J_0$ on $(-1,2)\times(\mathbb{R}/\mathbb{Z})\to (-1,2)\times(\mathbb{R}/\mathbb{Z})\times M$,   equal to  $Area((-1,2)\times(\mathbb{R}/\mathbb{Z}))+\int_{(-1,2)\times\mathbb{R}/\mathbb{Z}}|dw_{\infty}|^{2}$), with 
 any bubbles given by $J_0$-holomorphic spheres in fibers $\{(s_0,t_0)\}\times M$.  Meanwhile we have \[ \limsup\int_{(-1,2)\times (\mathbb{R}/\mathbb{Z})}|dW_n|^{2}_{\tilde{J}_n}\leq Area((-1,2)\times(\mathbb{R}/\mathbb{Z}))+e_0\] by hypothesis, so the fact that $e_0$ is strictly less than the minimal energy of a nonconstant $J_0$-holomorphic sphere implies that no bubbles can appear in the limit.  As such, we in fact have $w_n\to w_{\infty}$ in $L^{1,p}_{loc}$ for each $p<\infty$ (by Corollary 1.3 of \cite{IS}), and therefore also in $C^{\alpha}([0,1]\times (\mathbb{R}/\mathbb{Z});M)$ for each $0<\alpha<1$ by the Sobolev lemma.  By Fatou's Lemma, we have that \[\int_{(-1,2)\times(\mathbb{R}/\mathbb{Z})}|dw_{\infty}|^{2}_{J_0}\leq \liminf\int_{(-1,2)\times(\mathbb{R}/\mathbb{Z})}|dw_n|^{2}_{J_0}\leq e_0,\] and therefore that $w_{\infty}([0,1]\times (\mathbb{R}/\mathbb{Z}))$ has diameter at most $\beta/2$.  So since $w_n\to w_{\infty}$ in $C^0$ it follows that, for sufficiently large $n$, $w_n([0,1]\times(\mathbb{R}/\mathbb{Z}))$ has diameter at most $\beta$, in contradiction with the assumption that all $w_n([0,1]\times(\mathbb{R}/\mathbb{Z}))$ had diameter larger than $\beta$.  This contradiction proves the lemma.
\end{proof}

\begin{lemma}\label{toprest}  There are constants $\eta_1, e_1>0$ with the following property. Let $(H^s,\bar{J})$ be a monotone homotopy  from $(H^-,\bar{J})$ to $(H^{+},\bar{J})$, such that there are functions $f_s\co\mathbb{R}\to\mathbb{R}$ and $\rho_s\co 
(\mathbb{R}/\mathbb{Z})\times E(R)\to \mathbb{R}$, with \[ \|\rho_s\|_{C^1}<\eta_1 \mbox{ and } H^s(t,x)=f_s(L(x))+\rho_s(t,x).\]  Assume furthermore that $H^+$ and $H^-$ are both nondegenerate Hamiltonians, with the property that each one-periodic orbit $\gamma_{\pm}$ of $X_{H^{\pm}}$ lies in just one fiber (depending on $\gamma_{\pm}$) of the projection $\pi\co E(R)\to M$, so that in particular each $\gamma_{\pm}$ has a fiberwise capping $(w_0)_{\gamma_{\pm}}\co D^2\to E(R)$. 
Suppose that $u\co\mathbb{R}\times (\mathbb{R}/\mathbb{Z})\to E(R)$ is a solution to the Floer equation \begin{equation} \label{lemf}\frac{\partial u}{\partial s}+\bar{J}\left(\frac{\partial u}{\partial t}-X_{H^s}(t,u(s,t))\right)=0,\end{equation} such that $u(s,\cdot)\to \gamma_{\pm}$ as $s\to \pm\infty$.  Suppose that \[ [\gamma_+,(w_0)_{\gamma_-}\#u]=[\gamma_+,(w_0)_{\gamma_+}\#A].\] (where $A\in\frac{\pi_2(E(R))}{\ker\langle c_1(TE(R)),\cdot\rangle\cap\ker\langle [\omega],\cdot\rangle}$ and the notation signifies equivalence in $\tilde{\mathcal{L}}$).  Then \[ \mbox{ either }A=0\mbox{ or }\int_{\mathbb{R}\times (\mathbb{R}/\mathbb{Z})}\left|\frac{\partial u}{\partial s}\right|^2>e_1.\]  

\end{lemma}
\begin{remark} We emphasize that the constants $\eta_1,e_1$ are independent of the functions $f_s$.\end{remark}
\begin{proof}
Applying the linearization $\pi_*$ of the bundle map $\pi\co E(R)\to M$ to (\ref{lemf}) and using that $\pi_*\circ \bar{J}=J_0\circ \pi_*$ and that each $X_{f_s\circ L}$ is a vertical vector field, we obtain that \[ \frac{\partial}{\partial s}(\pi\circ u)+J_0\frac{\partial}{\partial t}(\pi\circ u)=J_0\pi_*X_{\rho_s}(u(s,t)).\]

Let $\beta$ be one half of the injectivity radius of the Riemannian manifold $(M,g_0)$.  This determines constants $\eta_0,e_0$ as in Lemma \ref{iscpct}.  Choose $\eta_1>0$ so that the assumption that $\|\rho_s\|_{C^1}<\eta_1$ 
implies that $\|J_0\pi_*X_{\rho_s}(u(s,t))\|_{C^0}<\eta_0$.
Choose $e_1>0$ so that the assumption that $\int_{\mathbb{R}\times (\mathbb{R}/\mathbb{Z})}\left|\frac{\partial u}{\partial s}\right|^2\leq e_1$  
implies that 
$\int_{\mathbb{R}\times(\mathbb{R}/\mathbb{Z})}\left|\frac{\partial (\pi\circ u)}{\partial s}\right|_{J_0}^{2}dsdt<e_0$.  Then for any $s_0\in\mathbb{R}$, Lemma \ref{iscpct} applied to the map $w\co (-1,2)\times (\mathbb{R}/\mathbb{Z})\to M$ defined by $w(s,t)=\pi\circ u(s-s_0,t)$ shows that, if $\|\rho_s\|_{C^1}<\eta_1$ 
and $\int_{\mathbb{R}\times (\mathbb{R}/\mathbb{Z})}\left|\frac{\partial u}{\partial s}\right|^2\leq e_1$, then the image of $\pi\circ u|_{[s_0,s_0+1]\times (\mathbb{R}/\mathbb{Z})}$ has diameter at most $\beta$.  In particular, the image of every circle $\{s\}\times(\mathbb{R}/\mathbb{Z})$ under $\pi\circ u$ has diameter at most $\beta$.  
  
Now the hypothesis on $u$ implies that $\pi\circ u(s,\cdot)$ converges to the constant map at some point $p_{\pm}\in M$ as $s\to \pm\infty$.  So identifying $S^2$ with \[ \frac{[-\infty,\infty]\times (\mathbb{R}/\mathbb{Z})}{(\pm\infty,t)\sim (\pm\infty,0)},\] $\pi\circ u$ extends to a map $\overline{\pi \circ u}\co S^2\to M$; this sphere is the image under $\pi$ of the sphere obtained by attaching the fiberwise cappings of $\gamma_{\pm}$ to the image of $u\co \mathbb{R}\times(\mathbb{R}/\mathbb{Z})\to E(R)$ at $s=\pm\infty$.  Since the diameter of the image of each $\{s\}\times(\mathbb{R}/\mathbb{Z})$ is less than the injectivity radius of $M$, we can write $\overline{\pi\circ u}(s,t)=exp_{\overline{\pi\circ u}(s,0)}(v(s,t))$ for some smoothly-varying tangent vectors $v(s,t)$ at $\overline{\pi\circ u}(s,t)$, with $v(\pm\infty,t)=0$ and $|v(s,t)|<\beta$.  Then writing $\overline{\pi\circ u}_{\sigma}(s,t)=exp_{\pi\circ u(s,0)}(\sigma v(s,t))$ ($0\leq \sigma\leq 1$) gives a homotopy of $\overline{\pi\circ u}$ to 
 the map $(s,t)\to \overline{\pi\circ u}(s,0)$, which is obviously nullhomotopic.  Thus $\pi\circ \bar{u}$ must have been nullhomotopic.  So since $\pi\circ \bar{u}$ is the image of the sphere in $E(R)$ obtained by gluing the fiberwise cappings of $\gamma_{\pm}$ to the ends of $u$, and since $\pi\co E(R)\to M$ induces an isomorphism on $\pi_2$, this latter sphere must be nullhomotopic.  In particular, in the notation of the lemma, this implies that $A=0$.  
\end{proof}

\begin{cor}\label{bigcor}  There are $\delta_1, e_1>0$ with the following property.  Let $H^{\ep}$ be a perturbation of a $\delta h$-Hamiltonian $H=f\circ L+\delta h\circ\pi$ as constructed before Proposition \ref{orbits} with $0<\delta<\delta_1$ and $\ep$ sufficiently small, and let $H^{s,\ep}$ $(-\infty\leq s\leq \infty)$ be a path of such perturbations of $\delta h$-Hamiltonians $f_s\circ L+\delta h\circ \pi$.  Assume that $f'(L),f'_s(L)$ are equal to $2\pi$ for $L\geq \alpha_0$, so that the Floer homologies $HF^{[a,b]}_{*}(H^{\ep})$, $HF^{[a,b]}_{*}(H^{\pm\infty,\ep})$ are defined for sufficiently small intervals $[a,b]$.  Suppose also that $a,b\in\mathbb{R}$ with \[ a<b<a+e_1.\]  Then \begin{itemize} \item For generic families of almost complex structures $J_t$ sufficiently close to $\bar{J}$, the boundary operator $\partial_{H,J_t}\co CF^{[a,b]}_{*}(H^{\ep})\to CF^{[a,b]}_{*}(H^{\ep})$ has the form \[ \partial_{H,J_t}[\gamma,(w_0)_{\gamma}\#A]=\sum_{\gamma'}a_{\gamma\gamma'}[\gamma',(w_0)_{\gamma'}\#A]\] for some numbers $a_{\gamma\gamma'}$.  \item If $(H^{s,\ep},J_{s,t})$ is a regular monotone homotopy with $J_{s,t}$ sufficiently $C^2$-close to the constant path at $\bar{J}$, then the chain map $\Phi_{H^{s,\ep},J_{s,t}}\co CF^{[a,b]}_{*}(H^{-\infty,\ep})\to CF^{[a,b]}_{*}(H^{\infty,\ep})$ has the form \[ \Phi_{H^{s,\ep},J_{s,t}}[\gamma_-,(w_0)_{\gamma_-}\#A]=\sum_{\gamma_+}a_{\gamma_-\gamma_+}[\gamma_+,(w_0)_{\gamma_+}\#A]\]
\end{itemize}
\end{cor}

\begin{proof} \emph{A priori}, the maps $\partial_{H,J_t}$ and $\Phi_{H^{s,\ep},J_{s,t}}$ have the form \[[\gamma,(w_0)_{\gamma}\#A]\mapsto \sum_{\gamma',B}a_{\gamma\gamma',B}[\gamma',(w_0)_{\gamma'}\#A\#B],\] where $a_{\gamma\gamma',B}$ counts cylinders $u$ satisfying the appropriate equation with $[\gamma',(w_0)_{\gamma}\#u]=[\gamma',(w_0)_{\gamma'}\#B]$; the content of the corollary is that the only nonzero 
$a_{\gamma\gamma',B}$ are those with $B=0$.  Since our restriction to the action window $[a,b]$ forces the $u$ being considered to have energy less than $e_1$, this is essentially an immediate consequence of Lemma \ref{toprest} (with the $H^s$ of Lemma \ref{toprest} set equal to $H^{\ep}$ independently of $s$ for the first part and to $H^{s,\ep}$ for the second part), except that Lemma \ref{toprest} concerned solutions to analogues of the Floer equations used to define $\partial_{H,J_t}$ and $\Phi_{H^{s,\ep},J_{s,t}}$ with $J_t$ and $J_{s,t}$ replaced by the nongeneric almost complex structure $\bar{J}$.  However, if our corollary were false then applying Gromov compactness to solutions of the relevant Floer equation using almost complex structures equal to $J_{t}^{n}$ or $J_{s,t}^{n}$ where $J_{t}^{n},J_{s,t}^{n}\to \bar{J}$ would yield a solution $u$ (possibly just one piece of a broken trajectory) to (\ref{lemf}), 
whose energy and topology would contradict Lemma \ref{toprest}.
\end{proof}

\section{Detecting periodic orbits}

We now turn to the proof of Theorem \ref{main2}.  Throughout this section, we \textbf{assume that either $c_1(TE(R))=0$ mod torsion or else that $h$ has no critical points of index $2m-1$}. Our strategy is to some extent modeled on that used in Section 6 of \cite{GG2} to prove the result in the spherically rational case, though of course the possible irrationality of the symplectic form will introduce additional subtleties.  Let $K\co E(R)\to\mathbb{R}$ be an autonomous Hamiltonian on $E(R)$ which attains a Morse--Bott nondegenerate minimum (say equal to $0$) along the zero section.  After possibly shrinking $R$, this implies that there are constants $C_1,C_2>0$ such that \begin{align*} C_1 L(x)&\leq K(x)\leq C_2 L(x),
\\ C_1L(x)&\leq dL_x(\nabla K(x))\leq C_2 L(x)\quad (x\in E(R));\end{align*} in particular, the only critical points of $K$ in $E(R)$ are the points on the zero section.  

Our intention is to produce a periodic orbit for $X_K$ on an energy level in an arbitrary open interval $(3\rho-\beta,3\rho+\beta)$ where $\rho>0$ is sufficiently small and $\beta\ll \rho$.   Put \[ C(\rho)=\frac{30\pi}{C_1}\rho .\]

Throughout the following, $\rho$ will be taken small enough that $\frac{30\rho}{C_1}<\alpha_0$ and $C(\rho)<\min\{e_1,D\alpha_{0}^{4}\}$, where $D$ and $\alpha_0$ are as in Theorem \ref{c0} and $e_1$ is as in Corollary \ref{bigcor}.

Where $\eta>0$ is a (small) real number to be specified later, let $g_{\rho,\beta}\co [0,\infty)\to [0,\infty)$ be a smooth function with the following properties: \begin{itemize} \item $g'_{\rho,\beta}(s)>0$ for all $s$;
\item $\eta/2<g'_{\rho,\beta}(s)<\eta$ if $s\notin [3\rho-\beta,3\rho+\beta]$;
\item $g(0)=0$ and $g(3\rho+\beta)=C(\rho)$.\end{itemize}

Define $H_0=g_{\rho,\beta}\circ K$.  Thus $\gamma\co [0,1]\to E(R)$ is a $1$-periodic orbit of $X_{H_0}$ if and only if $t\mapsto \gamma\left(\frac{t}{g_{\rho,\beta}'(\gamma(0))}\right)$ defines a $g_{\rho,\beta}'(\gamma(0))$-periodic orbit of $X_K$.  One property that the still-to-be-specified parameter $\eta$ will have is that $X_K$ has no nonconstant periodic orbits of period at most $2\eta$ (such an $\eta$, depending only on the $C^2$-norm of $K$, exists by the Yorke estimate \cite{Y}), so this means that any nonconstant $1$-periodic orbit of $X_{H_0}$ will correspond to a periodic orbit of $X_K$ at some energy level in $(3\rho-\beta,3\rho+\beta)$.

Introduce two smooth strictly increasing functions $f_a,f_b\co [0,\infty)\to [0,\infty)$, depending on $\rho$ but not on $\beta$, having the following properties.  (see Figure 1; `$b$' stands for ``below,'' and `$a$' for ``above''):  \begin{itemize}  \item For $s\in [0,\rho/C_2]$, $f_a(s)=C_2\eta s$; \item For $s\in [0,7\rho/(2C_1)]$, $f_b(s)=C_1\eta s/2$. \item $f_a\left(\frac{2\rho}{C_2}\right)=f_b\left(\frac{4\rho}{C_1}+\frac{\rho}{C_2}\right)=C(\rho)$.
\item $f''_a(s)>0$ for $\frac{\rho}{C_2}<s<\frac{4\rho}{3C_2}$, $f''_a(s)<0$ for $\frac{5\rho}{3C_2}<s<\frac{2\rho}{C_2}$. \item For $\frac{4\rho}{3C_2}\leq s\leq \frac{5\rho}{3C_2}$, $f'_a(s)$ is not an integer multiple of $4\pi$.\item
For $4\rho/C_1<s<4\rho/C_1+\rho/C_2$ we have $f_b(s)=f_a\left(s-\left(\frac{4\rho}{C_1}-\frac{\rho}{C_2}\right)\right)$
\item For $2\rho/C_2<s<20\rho/C_1$ we have $0<f'_a(s)\leq C_2\eta$, and for  $4\rho/C_1+\rho/C_2<s<20\rho/C_1$ we have $0<f'_b(s)\leq C
_2\eta$.
\item For $5\rho/(2C_2)<s<10\rho/C_1$ we have $f'_a(s)=C_2\eta$, and for $4\rho/C_1+\rho/C_2<s<10\rho/C_1$ we have $f'_b(s)=C_1\eta/2$.
\item $f_b(s)\leq f_a(s)$ for all $s$, and $f_a(s)=f_b(s)$ for $s>20\rho/C_1$.
\item For $s>30\rho/C_1$, $f'_a(s)=f'_b(s)=2\pi$.\end{itemize}

\begin{figure}
\includegraphics[scale=0.6]{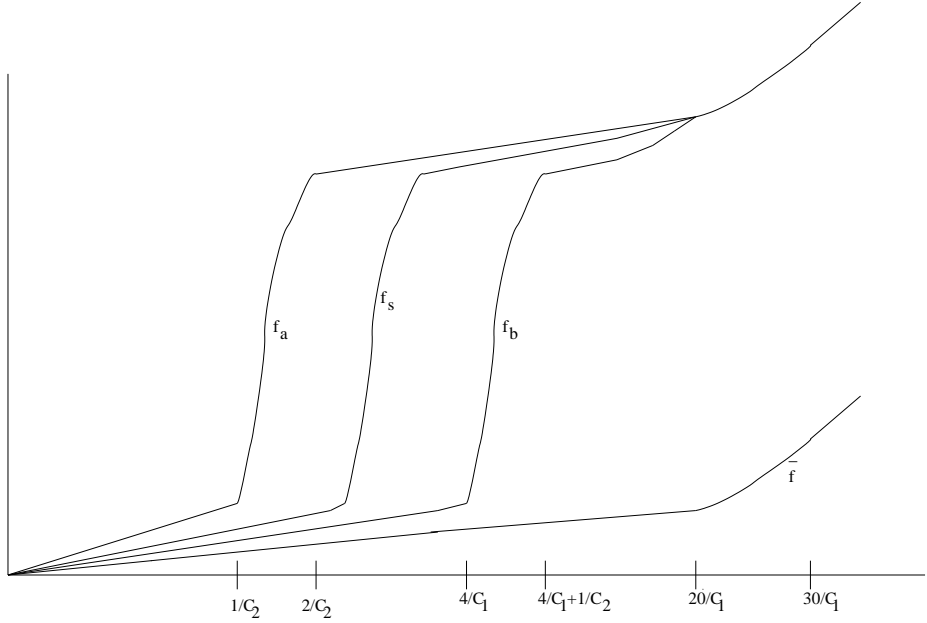}
\caption{Graphs of the functions $f_a, f_b$, $f_s$ (for typical $s\in [-1,1]$; see proof of Proposition \ref{p2}), and $\bar{f}$ (see proof of Proposition \ref{p3}).  The horizontal axis, which is not to scale, is labeled in units of $\rho$.}
\label{bundlegraphs}
\end{figure}

The defining properties of the constants $C_1$ and $C_2$, together with the above properties of $f_a,f_b$, and $g_{\rho,\ep}$, imply in particular that \[ f_b(L(x))\leq g_{\rho,\ep}(K(x))\leq f_a(L(x))\mbox{ when }L(x)\leq 10\rho/C_1.\] 
Choose a smooth, monotone increasing function $\zeta\co \mathbb{R}\to [0,1]$ such that $\zeta(s)=0$ for $s\leq 5\rho/C_1$ and $\zeta(s)=1$ for $s\geq 10\rho/C_1$ and define \[ \tilde{H}_0(x)=(1-\zeta(L(x)))H_0(x)+\zeta(L(x))\left(\frac{f_a(L(x))+f_b(L(x))}{2}\right).\]
Thus $f_b\circ L\leq \tilde{H}_0\leq f_a\circ L$, with equality in the region where $L\geq 20\rho/C_1$.

Now outside the region where $3\rho-\beta\leq K(x)\leq 3\rho+\beta$, the  Hamiltonian vector field $X_{\tilde{H}_0}$ coincides with a vector field on $E(R)$ whose $C^1$-norm is bounded by an $\rho$-dependent constant times $\eta$.  We now specify $\eta$ by requiring that this bound on the $C^1$-norm be less than one-half times the minimal $C^1$-norm of a vector field on $E(R)$ having a nonconstant periodic orbit of period at most $1$, as given by the Yorke estimate.  Thus any periodic orbit of $X_{\tilde{H}_0}$ with period at most $2$ intersects the region $\{3\rho-\beta<K<3\rho+\beta\}$; in fact, since $X_{\tilde{H}_0}$ is tangent to each of the level surfaces $K^{-1}(\{y\})$ with $y<5\rho$, any such orbit must be entirely contained in $\{3\rho-\beta<K<3\rho+\beta\}$.  We also require that $\eta$ be small enough that  $f_a$ and $f_b$ as constructed above have the property that all nonconstant periodic orbits of $f_a\circ L$ (resp. $f_b\circ L$) with period less than $2$ are contained in the region where $\rho/C_2\leq L(x)\leq 2\rho/C_2$ (resp. $4\rho/C_1\leq L(x)\leq 4\rho/C_1+\rho/C_2$). 

For sufficiently small $\delta>0$, then, the Hamiltonians \[ F_b=f_b\circ L+\delta h\circ \pi,\quad F_a=f_a\circ L+\delta h\circ \pi,\quad H=\tilde{H}_0+\delta h\circ \pi \] are $\delta h$-Hamiltonionians, nondegenerate perturbations of which have Floer homologies $HF^{[c,d]}_{*}$ for sufficiently small intervals $[c,d]$.  Consider perturbations $F_{b}^{\ep},F_{a}^{\ep}$ of $F_a,F_b$  as constructed in the previous section, for sufficiently small $\ep$.  Since $F_{b}^{\ep}$ differs from $F_b$ only where $4\rho/C_1\leq L(x)\leq 4\rho/C_1+\rho/C_2$, and likewise $F_{a}^{\ep}$ differs from  $F_a$ only where $\rho/C_2\leq L(x)\leq 2\rho/C_2$, we have $F_{b}^{\ep}\leq H\leq F_{a}^{\ep}$ everywhere for $\ep$ sufficiently small.  Also (for sufficiently small $\delta$), all $1$-periodic orbits of $H$ are nondegenerate except for possibly those contained in the region where $3\rho-\beta\leq K\leq 3\rho+\beta$ (which in turn is contained in the region where $\frac{5\rho}{2C_2}\leq L\leq \frac{7\rho}{2C_1}$), so there are arbitrarily small nondegenerate perturbations $H'$ of $H$ still having the property that \[ F_{b}^{\ep}\leq H'\leq F_{a}^{\ep}.\] (We'll always assume that $H'-H$ is supported in $\{3\rho-\beta\leq K\leq 3\rho+\beta\}$.)  So, for $d-c$ sufficiently small, we have a commutative diagram \begin{equation}\label{diag} \xymatrix{ HF^{[c,d]}_{*}(F_{b}^{\ep})\ar[rr]^{\Psi_{F_{b}^{\ep}}^{F_{a}^{\ep}}} \ar[dr]_{\Psi_{F_{b}^{\ep}}^{H'}} & & HF^{[c,d]}_{*}(F_{a}^{\ep})\\ & HF^{[c,d]}_{*}(H')\ar[ur]_{\Psi_{H'}^{F_{a}^{\ep}}} & }.\end{equation}

The construction of $f_a$ and $f_b$ shows that there are unique real numbers $\ell_a,\ell_b$ with the property that \[ f'_a(\ell_a)=f'_b(\ell_b)=4\pi, \quad f''_a(\ell_a)<0,\, f''_b(\ell_b)<0.\]  In fact, we will have \[ \frac{5\rho}{3C_2}<\ell_a<\frac{2\rho}{C_2},\quad \ell_b=\ell_a+\left(\frac{4\rho}{C_1}-\frac{\rho}{C_2}\right), \quad f_a(\ell_a)=f_b(\ell_b).\]  Also, since $f_a(\frac{2\rho}{C_2})=C(\rho)=\frac{30\pi \rho}{C_1}$,  $f'_a(\ell)<4\pi$ for $\ell_a<\ell<\frac{2\rho}{C_2}$, and $\ell_a\geq \frac{2\rho}{C_2}-\frac{\rho}{3C_2}$ it follows that \[ f_b(\ell_b)=f_a(\ell_a)\geq C(\rho)-\frac{4\pi \rho}{3C_2}=\pi \rho\left(\frac{30}{C_1}-\frac{4}{3C_2}\right).\]

Order the critical points $p_i$ of the Morse function $h$ so that $p_1$ is the unique local maximum of $h$.  Adapting the notation of the previous section, $F_{a}^{\ep}$ and $F_{b}^{\ep}$ then have periodic orbits $\gamma_{1,\ell_a}^{-}, \gamma_{1,\ell_b}^{-}$ lying in the fiber over $p_1$.  Proposition \ref{grading} then gives that, where as usual $w_0$ denotes the fiberwise capping, \[ \mu_{F_{a}^{\ep}}([\gamma_{1,\ell_a}^{-},w_0])=\mu_{F_{b}^{\ep}}([\gamma_{1,\ell_b}^{-},w_0])=2n-2m=2r.\]

We also see that \[ \mathcal{A}_{F_{a}^{\ep}}([\gamma_{1,\ell_a}^{-},w_0])=4\pi\ell_a-f(\ell_a)+O(\ep)\in \left(\frac{-30\pi \rho}{C_1},-\frac{20\pi \rho}{C_1}\right)\] (assuming $\ep$ to be sufficiently small)
while $\mathcal{A}_{F_{b}^{\ep}}([\gamma_{1,\ell_b}^{-},w_0])=\mathcal{A}_{F_{a}^{\ep}}([\gamma_{1,\ell_a}^{-},w_0])+4\pi(\ell_b-\ell_a)+O(\ep)$, so since $0<l_b-l_a<\frac{4\rho}{C_1}$ we see that \[ 
  \mathcal{A}_{F_{b}^{\ep}}([\gamma_{1,\ell_b}^{-},w_0])\in \left(-\frac{30\pi \rho}{C_1},-\frac{4\pi \rho}{C_1}\right).\]
  
  So let \[ c(\rho)=-C(\rho)=-\frac{30\pi \rho}{C_1},\quad d(\rho)=-\frac{4\pi \rho}{C_1}.\]  In particular, $d(\rho)-c(\rho)=\frac{26\pi \rho}{C_1}$ is less than the parameter $e_1$ of Corollary \ref{bigcor}; below, we always assume that the perturbing parameters $\delta$ and $\ep$ are sufficiently small.  The main result of this section is:
  
\begin{prop} \label{p3}  There is a $1$-periodic orbit $\gamma$ of $X_H$ which is contained in the region $\{3\rho-\beta\leq K\leq 3\rho+\beta\}$ and a map $w\co D^2\to E(R)$ with $w|_{\partial D^2}=\gamma$ such that the Salamon--Zehnder index $\Delta([\gamma,w],H)$ satisfies $-2r\leq \Delta([\gamma,w],H)\leq 2m+1$.  
\end{prop}


Theorem \ref{main2} follows quickly from Proposition \ref{p3}: recall that $H$ was taken to have the form $H=\tilde{H}_0+\delta h\circ \pi$ where $\delta>0$ was arbitrarily small, so applying the Arzel\`a-Ascoli theorem to a sequence of orbits $\gamma_{\delta_k}$ (with capping discs $w_{\delta_k}$) given by Proposition \ref{p3} where $\delta_k\searrow 0$ gives a $1$-periodic orbit $\gamma$ for $X_{\tilde{H_0}}$ in the region $\{3\rho-\beta\leq K\leq 3\rho+\beta\}$.  But in this region we have $\tilde{H}_0=g_{\rho,\beta}\circ K$, so evidently $\gamma$ is a $1$-periodic orbit for $X_{g_{\rho,\beta}\circ K}$, which in turn implies that $X_K$ has a periodic orbit (of some, possibly large, period, but with the same Salamon--Zehnder index by Lemma 2.6 of \cite{GG2}) contained in the region $\{3\rho-\beta\leq K\leq 3\rho+\beta\}$  (for every sufficiently small parameter $\beta$). By the continuity of the Salamon--Zehnder invariant the periodic orbit $\gamma$ of $X_K$ that we have obtained has a capping disc $w$ (obtained by gluing a thin cylinder from $\gamma_{\delta_k}$ to $\gamma$ to the end of $w_{\delta_k}$ where $\gamma_{\delta_k}$ is close to $\gamma$)  with $\Delta([\gamma,w],K)\in [-2r,2m+1]$.  This proves Theorem \ref{main2} (interchanging our $3\rho$ and $\beta$ with $\rho$ and $\ep/2$ in the notation of Theorem \ref{main}) for the case that $(P,\Omega)=(E(R),\omega)$ and hence, using the Weinstein neighborhood theorem as in the introduction, for arbitrary $(P,\Omega)$.

It remains to prove Proposition \ref{p3}.  The proof depends on the following propositions (recall that $p_1$ is the unique local maximum of $h$):

\begin{prop}\label{p1} The Floer chains $[\gamma_{1,\ell_b}^{-},w_0],[\gamma_{1,\ell_a}^{-},w_0]$ are  cycles representing nontrivial homology classes in their respective Floer complexes $CF^{[c(\rho),d(\rho)]}_{*}(F_{b}^{\ep})$, $CF^{[c(\rho),d(\rho)]}_{*}(F_{a}^{\ep})$.  In fact, if $J_t$ is a generic path of almost complex structures sufficiently close to $\bar{J}$, then we have \[ \partial_{F_{a}^{\ep},J_t} [\gamma_{1,\ell_a}^{-},w_0]=\partial_{F_{b}^{\ep},J_t} [\gamma_{1,\ell_b}^{-},w_0]=0\] and, for every generator $[\gamma,w]$ of $CF^{[c(\rho),d(\rho)]}_{2r+1}(F_{a}^{\ep})$ or $CF^{[c(\rho),d(\rho)]}_{2r+1}(F_{b}^{\ep})$, \[ \langle \partial_{F_{a}^{\ep},J_t}[\gamma,w],[\gamma_{1,\ell_a}^{-},w_0]\rangle=0,\quad \langle \partial_{F_{b}^{\ep},J_t}[\gamma,w],[\gamma_{1,\ell_b}^{-},w_0]\rangle=0.\]
\end{prop}

\begin{prop}\label{p2}  For a generic family of almost complex structures $J_t$ near $\bar{J}$, there is a regular monotone homotopy $(F^{\ep}_{s},\bar{J}_{s,t})$ from $(F^{\ep}_{b},J_t)$ to $(F^{\ep}_{a},J_t)$ such that the chain map \[ \Phi_{F_{s}^{\ep},\bar{J}_{s,t}}\co CF^{[c(\rho),d(\rho)]}_{*}(F_{b}^{\ep})\to CF^{[c(\rho),d(\rho)]}_{*}(F_{a}^{\ep})\] has \[ 
\Phi_{F_{s}^{\ep},\bar{J}_{s,t}}([\gamma_{1,\ell_b}^{-},w_0])=[\gamma_{1,\ell_a}^{-},w_0].\]
\end{prop}

\subsection{Proof of Proposition \ref{p1}}  We prove the result for the generator  $[\gamma_{1,\ell_a},w_0]$ of  $CF^{[c(\rho),d(\rho)]}(F_{a}^{\ep})$; the proof when `$a$' is replaced by `$b$' is identical.  (The proof will also apply equally well to the Hamiltonians $F_{s}^{\ep}$ introduced in the proof of Proposition \ref{p2}.)

Where $\partial=\partial_{F_{a}^{\ep},J_t}$ for a generic family of almost complex structures $J_t$ near $\bar{J}$, Corollary \ref{bigcor} shows that $\partial [\gamma_{1,\ell_a}^{-},w_0]$ has the form $\sum_{\gamma}a_{\gamma}[\gamma,w_0]$ where $\mu_{F_{a}^{\ep}}([\gamma,w_0])=2r-1=2n-2m-1$.  But Proposition \ref{grading} shows that there are no periodic orbits $\gamma$ with $\mu_{F_{a}^{\ep}}([\gamma,w_0])=2n-2m-1$, so evidently \[ \partial [\gamma_{1,\ell_a}^{-},w_0]=0.\]

We now need to show that $ [\gamma_{1,\ell_a}^{-},w_0]$ never appears with nonzero coefficient in any boundary. By Corollary \ref{bigcor}, the only generators $[\gamma,w]$ for which $\partial [\gamma,w]$ can contain $ [\gamma_{1,\ell_a}^{-},w_0]$ with nonzero coefficient are those with $w=w_0$ and $\mu_{F_{a}^{\ep}}([\gamma,w_0])=2r+1$.  The only possibilities for such a $\gamma$ are, by Proposition \ref{grading} and the construction of $f_a$: \begin{itemize} \item[(i)] $\gamma=\gamma_{j,0}$, the constant orbit at a critical point $p_j$ of $h$ on the zero section, where $ind_{p_j}h=2m-1$; 
\item[(ii)] $\gamma=\gamma_{j,\ell_a}^{-}$, where $ind_{p_j}h=2m-1$;
\item[(iii)] $\gamma=\gamma_{1,\ell'_a}^{-}$, where $\ell'_a$ is the unique real number with $f'_a(\ell'_a)=4\pi$ and $f''_a(\ell'_a)>0$;
\item[(iv)] (only in the case $r=1$) $\gamma=\gamma_{1,\ell_a}^{+}$.\end{itemize}

Now we have $\mathcal{A}_{F_{a}^{\ep}}([\gamma_{j,0},w_0])=O(\delta)\notin [c(\rho),d(\rho)]$ since $d(\rho)=-\frac{4\pi \rho}{C_1}$.  Also \[ \mathcal{A}_{F_{a}^{\ep}}([\gamma_{1,\ell'_a}^{-},w_0])=4\pi \ell'_a-f_a(\ell'_a)+O(\delta,\ep)>0 \] since the construction of $f_a$ ensures that $f'_a(\ell)<4\pi$ for $\ell<\ell'_a$, so that \linebreak[6]
$\mathcal{A}_{F_{a}^{\ep}}([\gamma_{1,\ell'_a}^{-},w_0])\notin [c(\rho),d(\rho)]$.  So in fact generators of types (ii) and (iv) above are the only ones that actually belong to $CF^{[c(\rho),d(\rho)]}_{2r+1}(F_{a}^{\ep})$.

Assume first that we are in the case where $c_1(TE(R))$ is a torsion class.  In this case, we can use the results discussed in the appendix at the end of this paper to deduce that no  generators of type (ii) or (iv) can have boundary containing $[\gamma_{1,\ell_a}^{-},w_0]$ with nonzero coefficient.  Indeed, in the notation of the appendix, let $N$ denote the sphere bundle in $E(R)$ over $M$ of radius $\ell_a$, and let the $(\mathbb{R}/\mathbb{Z})$-action be given by $t\cdot x=e^{-2\pi it}x$.  $F_{a}^{\ep}$ is then a $C^2$-small perturbation of $f_a\circ L$ (which plays the role of the Hamiltonian $H_0$ in the appendix), so for sufficiently small $\delta$ and $\ep$ the Hamiltonian $F_{a}^{\ep}$ can be used to compute the local Floer homology $HF^{loc}_{*}(f_a\circ L,\mathcal{U})$ ($\mathcal{U}$ is, as in the appendix, a suitable small neighborhood in the contractible loopspace of the set of $1$-periodic orbits which foliate $N$).  Write $\partial^{loc}$ for the differential on the local Floer complex $CF^{loc}_{*}(F_{a}^{\ep},f_a\circ L,\mathcal{U})$ (as in the appendix, this complex is generated by periodic orbits of $X_{F_{a}^{\ep}}$ which are contained in $\mathcal{U}$, with no ``capping data'').   Suppose that $\gamma$ is an orbit of type (ii) or (iv) above.  We then see that any of the $[\gamma,w_0]$ has action differing from that of  $[\gamma_{1,\ell_a}^{-},w_0]$ by an amount tending to zero with $\delta$ and $\ep$, so  that any cylinder $u\co \mathbb{R}\times(\mathbb{R}/\mathbb{Z})\to E(R)$  contributing to the matrix element  \[\langle \partial[\gamma,w_0],[\gamma_{1,\ell_a}^{-},w_0]\rangle \]has its energy tending to $0$ as $\ep,\delta\to 0$.  So a Gromov compactness argument shows that, for $\ep,\delta$ sufficiently small, each $u(s,\cdot)$ belongs to $\mathcal{U}$, and so $u$ contributes to $\langle \partial^{loc}\gamma,\gamma_{1,\ell_a}^{-}\rangle.$  Conversely, if $u$ contributes to $\langle \partial^{loc}\gamma,\gamma_{1,\ell_a}^{-}\rangle$, then 
Lemma \ref{toprest} implies that $[\gamma_{1,\ell_a}^{-},w_0\#u]=[\gamma_{1,\ell_a}^{-},w_0]$, so $u$ contributes to $\langle \partial[\gamma,w_0],[\gamma_{1,\ell_a}^{-},w_0]\rangle$.  Since the counting prescriptions for these $u$ are identical whether they are considered as contributing to $\partial$ or to $\partial^{loc}$, this shows that \[ 
\langle \partial[\gamma,w_0],[\gamma_{1,\ell_a}^{-},w_0]\rangle=\langle \partial^{loc}\gamma,\gamma_{1,\ell_a}^{-}\rangle \]
for every orbit $\gamma$ of type (ii) or (iv).  

But by Theorem \ref{morsebott}, $HF^{loc}_{*}(f_a\circ L,\mathcal{U})$ is equal as a relatively $\mathbb{Z}$-graded group to the singular homology of the sphere bundle $N$.  (The relative grading is by $\mathbb{Z}$ because of our assumption on $c_1$.)  Now the (strictly) largest relative grading difference $gr(\gamma,\gamma')$ between any two generators of $CF^{loc}_{*}(F_{a}^{\ep},f_a\circ L,\mathcal{U})$ is attained where $\gamma=\gamma_{j,\ell_a}^{+}$ and $\gamma'=\gamma_{1,\ell_a}^{-}$ where $j$ is chosen so that $p_j$ is a local minimum of $h$, and this relative grading difference is $\dim N=2m+2r-1$.  So in order for $HF^{loc}_{*}(f_a\circ L,\mathcal{U})$ to have $\mathbb{Z}_2$-summands in gradings separated by $\dim N$, it must be that some element of $CF^{loc}_{*}(F_{a}^{\ep},f_a\circ L,\mathcal{U})$ having the same relative grading as  $\gamma_{1,\ell_a}^{-}$ represents a nontrivial class in  $HF^{loc}_{*}(f_a\circ L,\mathcal{U})$.  But since $\gamma_{1,\ell_a}^{-}$ is the 
 only generator for $CF^{loc}_{*}(F_{a}^{\ep},f_a\circ L,\mathcal{U})$ in its relative grading, this is possible only if, for all $\gamma$, $\langle\partial^{loc}\gamma,\gamma_{1,\ell_a}^{-}\rangle=0$.  So if $\gamma$ is an orbit of type (ii) or (iv) we have 
\[ \langle \partial[\gamma,w_0],[\gamma_{1,\ell_a}^{-},w_0]\rangle=0,\] which completes the proof that 
$[\gamma_{1,\ell_a}^{-},w_0]$ has the stated properties.

We now turn to the alternate case, where we make no assumption on $c_1$ but we do assume that $h$ has no critical points of index $2m-1$.  This assumption obviously eliminates orbits of type (ii), so the only orbit to worry about is $\gamma=\gamma_{1,\ell_a}^{+}$ (in the case $r=1$; if $r\neq 1$ the proof is now complete).  The corresponding matrix element for $\partial$ can then likewise be shown to vanish by a local Floer homology argument: here we view $F_{a}^{\ep}$ as a $C^2$-small perturbation of $H_0=f_a\circ L+\delta h\circ \pi$ and use for $N$ the circle of radius $l_a$ in the fiber over the maximum $p_1$ of $h$.  $CF^{loc}_{*}(F_{a}^{\ep}, f_a\circ L+\delta h\circ \pi,\mathcal{U})$ then has $\gamma_{1,\ell_a}^{+}$ and $\gamma_{1,\ell_a}^{-}$ as its only two generators and has homology equal (at least as a relatively $(\mathbb{Z}/2)$-graded group) to the singular homology of the circle, in view of which $\gamma_{1,\ell_a}^{-}$ represents a nontrivial class in local Floer homology and an argument identical to that used in the $c_1=0$ case then  proves the result.

\subsection{Proof of Proposition \ref{p2}} We begin with a few more general considerations.  If $G\co (\mathbb{R}/\mathbb{Z})\times E(R)\to \mathbb{R}$ is a nondegenerate Hamiltonian and  if $\mu\in\mathbb{R}$ then the Hamiltonian vector fields of the Hamiltonians $G$ and $G+\mu$ coincide.  Hence their periodic orbits also coincide, while for any $[\gamma,w]\in\tilde{\mathcal{L}}$ we have $\mathcal{A}_{G+\mu}([\gamma,w])=\mathcal{A}_G([\gamma,w])-\mu$.  Assume that $G$ has the form (\ref{hform}) and that $a<b$  with $b-a$ small enough as to allow for the construction of $CF_{*}^{[a,b]}(G)$ as in Section 2.2.   Then there is a tautological identification \begin{equation} \label{taut}CF_{*}^{[a,b]}(G+\mu)\cong CF_{*}^{[a+\mu,b+\mu]}(G).\end{equation} 
\begin{lemma}\label{p2lemma} Let $G,G'\co (\mathbb{R}/\mathbb{Z})\times E(R)\to \mathbb{R}$ be two Hamiltonians of the form 
(\ref{hform}), let $\mu>0$, and suppose that, for all $t\in\mathbb{R}/\mathbb{Z}$ and all $x\in E(R)$, we have \[ G(t,x)\leq G'(t,x)\leq G(t,x)+\mu.\] If $J_t\in\mathcal{J}^{reg}(G)\cap \mathcal{J}^{reg}(G')$, let $(H^{s}_{1},J_{s,t})$ be a regular monotone homotopy from $(G,J_t)$ to $(G',J_t)$, and let $(H^{s}_{2},J'_{s,t})$ be a regular monotone homotopy from $(G',J_t)$ to $(G+\mu,J_t)$.  Then, with respect to the identification (\ref{taut}), the map \[ \Phi_{H^{s}_{2},J'_{s,t}}\circ\Phi_{H^{s}_{1},J_{s,t}}\co CF_{*}^{[a,b]}(G)\to  CF_{*}^{[a+\mu,b+\mu]}(G)\] is chain homotopic to the homomorphism \[ \Pi_{[a+\mu,b]}^{[a,b]}\co  CF_{*}^{[a,b]}(G)\to  CF_{*}^{[a+\mu,b+\mu]}(G)\] defined  on the generators $[\gamma,w]$ of $ CF_{*}^{[a,b]}(G)$ by \[ 
\Pi_{[a+\mu,b]}^{[a,b]}([\gamma,w])=\left\{\begin{array}{ll} [\gamma,w] & a+\mu\leq \mathcal{A}_G([\gamma,\mu])\leq b \\ 0 & a\leq \mathcal{A}_G([\gamma,\mu])<a+\mu.\end{array}\right.\]

\end{lemma}

\begin{proof} As a result of the definition of a monotone homotopy, if $S>0$ is large enough, one obtains a smooth path $(H^{s}(S),J_{s,t}(S))$ by setting $(H^{s}(S),J_{s,t}(S))$ equal to $(G',J_t)$ for $|s|<S$, to $(H^{s+2S}_{1},J_{s+2S,t})$ for $s\leq -S$, and to $(H^{s-2S}_{2},J'_{s-2S,t})$ for $s>S$.  $(H^{s}(S),J_{s,t}(S))$  is then a monotone homotopy from $(G,J_t)$ to $(G+\mu,J_t)$, and a standard gluing argument shows that, once $S>0$ is sufficiently large, this monotone homotopy is regular and has the property that the induced map \[ \Phi_{H^{s}(S),J_{s,t}(S)}\co CF_{*}^{[a,b]}(G)\to CF_{*}^{[a,b]}(G+\mu)\cong CF^{[a+\mu,b+\mu]}_{*}(G)\] is equal to $\Phi_{H^{s}_{2},J'_{s,t}}\circ\Phi_{H^{s}_{1},J_{s,t}}$.  

A different monotone homotopy from $(G,J_t)$ to $(G+\mu,J_t)$ is given by $(G+\mu\chi(s),J_t)$, where $\chi\co \mathbb{R}\to [0,1]$ is a smooth monotone increasing function with $\chi(s)=0$ for $s<-1$ and $\chi(s)=1$ for $s>1$.  Since the Hamiltonian vector field of $G+\mu\chi(s)$ is independent of $s$, the map $\Phi_{G+\mu\chi(s),J_t}$ induced by this monotone homotopy counts index $0$ solutions $u\co \mathbb{R}\times(\mathbb{R}/\mathbb{Z})\to E(R)$ to the usual Floer equation $\frac{\partial u}{\partial s}+J_t\left(\frac{\partial u}{\partial t}-X_G(t,u(s,t))\right)=0.$  But since $J_t\in  \mathcal{J}^{reg}(G)$ the only index-zero solutions to this equation have $u(s,t)=\gamma(t)$ for some one-periodic orbit $\gamma$ of $X_G$ (furthermore, the nondegeneracy of $G$ ensures that the linearizations of the Floer equation at these solutions are surjective, so that $(G+\mu\chi(s),J_t)$ is a regular monotone homotopy).  So the matrix elements $\langle \Phi_{G+\mu\chi(s),J_t}([\gamma,w]),[\gamma',w']\rangle$ are equal to $1$ whenever $[\gamma,w]=[\gamma',w']\in CF_{*}^{[a,b]}(G)\cap CF_{*}^{[a,b]}(G+\mu)$ and are equal to $0$ otherwise.  
But this is equivalent to saying that $\Phi_{G+\mu\chi(s),J_t}=\Pi_{[a+\mu,b]}^{[a,b]}$.

Since $(H^{s}(S),J_{s,t}(S))$ ($S\gg 0$) and $(G+\mu\chi(s),J_t)$ are both monotone homotopies from $(G,J_t)$ to $(G+\mu,J_t)$, the induced maps $\Phi_{H^{s}(S),J_{s,t}(S)}$ and $\Phi_{G+\mu\chi(s),J_t}$ are chain homotopic by a standard argument involving a homotopy connecting these two monotone homotopies.  So the lemma follows from the facts that $\Phi_{H^{s}(S),J_{s,t}(S)}=\Phi_{H^{s}_{2},J'_{s,t}}\circ\Phi_{H^{s}_{1},J_{s,t}}$ and $\Phi_{G+\mu\chi(s),J_t}=\Pi_{[a+\mu,b]}^{[a,b]}$.

(See, \emph{e.g.}, Section 6 of \cite{SZ} for the arguments that we have referred to as standard in this proof.)
\end{proof}

We now turn to the body of the proof of Proposition \ref{p2}.  
  The functions \linebreak[6] $f_a,f_b\co [0,\infty)\to [0,\infty)$ are surjective and strictly increasing; hence they are invertible.  Let $\phi_a,\phi_b\co [0,\infty)\to [0,\infty)$ denote their inverses (with respect to composition, of course).  Let $\chi\co \mathbb{R}\to [0,1]$ be a smooth, monotone increasing function with $\chi(s)=0$ for $s\leq -1$ and $\chi(s)=1$ for $s\geq 1$, and define $\phi_s=\chi(s)\phi_a+(1-\chi(s))\phi_b$; for any $s$ this is a strictly increasing function from $[0,\infty)$ onto itself, so let $f_s$ be the inverse of $\phi_s$.  We then have $f_a\leq f_s\leq f_b$, $f_s=f_a$ for $s\geq 1$, $f_s=f_b$ for $s\leq 1$, and $\frac{\partial f_s}{\partial s}\geq 0$. 
For $F_{s}^{\ep}$ we take the $s$-parametrized family $(f_s\circ L+\delta h\circ\pi)^{\ep}$ (where $\ep$ is appropriately small and the notation means that we perturb $f_s\circ L+\delta h\circ \pi$ using the parameter $\ep$ as in the previous section).  

Now since in the only regions where its derivative ever approaches an integer multiple of $4\pi$ the function $f_s$ has a graph which is just a translated-to-the-right-version of the graph of $f_a$ in the corresponding region, the $1$-periodic orbits of any given $X_{F_{s}^{\ep}}$ correspond precisely to the one-periodic orbits of $X_{F_{a}^{\ep}}$ (but with generally larger values of the parameter $\ell$).  In particular, each $F_{s}^{\ep}$ is nondegenerate a nondegenerate Hamiltonian and has each $1$-periodic orbit contained in a fiber over a critical point of $h$, and its only nonconstant $1$-periodic orbit $\gamma$ such that $\mu_{F_{s}^{\ep}}([\gamma,w_0])=2r$ is the orbit $\gamma=\gamma_{1,\ell(s)}^{-}$, where $\ell(s)$ decreases from $\ell_b$ when $s\leq -1$ to $\ell_a$ when $s\geq 1$.  

Let $\mu>0$ have the property that $c(\rho)+\mu<\mathcal{A}_{F_{a}^{\ep}}([\gamma_{1,\ell_a}^{-},w_0])$.  For any $s$, the proof of Proposition \ref{p1} applies just as well to $F^{\ep}_{s}$ and $[\gamma_{1,\ell(s)}^{-},w_0]$ as it does to $F_{a}^{\ep}$ and $[\gamma_{1,\ell_a}^{-},w_0]$.  Therefore, for each $s$ and for any $J'_t\in \mathcal{J}^{reg}(F^{\ep}_{s})$, we have $\partial_{F^{\ep}_{s},J'_t}^{[c(\rho),d(\rho)]}[\gamma_{1,\ell(s)}^{-},w_0]=0$ and, for any $[\gamma',w']\in CF^{[c(\rho),d(\rho)]}_{2r+1}(F^{\ep}_{s})$, \[ \langle \partial_{F^{\ep}_{s},J'_t}[\gamma',w'],[\gamma_{1,\ell(s)}^{-},w_0]\rangle=0.\]  

  Choose \[ -1=s_1<s_1<\cdots<s_{N+1}=1\mbox{ such that }0\leq F_{s_{j+1}}^{\ep}(t,x)-F_{s_{j}}^{\ep}(t,x)<\mu \] for all $j,t,x$, and let $J_t\in \cap_{j=1}^{N+1}\mathcal{J}^{reg}(F_{s_j}^{\ep})$.  
Let $(G_{j,1}^{s},J_{j,s,t})$ be a regular monotone homotopy from $(F_{s_{j}}^{\ep},J_t)$ to $(F_{s_{j+1}}^{\ep},J_t)$ and let $(G_{j,2}^{s},J'_{j,s,t})$ be a regular monotone homotopy from $(F_{s_{j+1}}^{\ep},J_t)$ to $(F_{s_j}^{\ep}+\mu,J_t)$, with $G_{j,1}^{s}$ and $G_{j,2}^{s}$ chosen to be of the same form as the $H^{s,\ep}$ in Corollary \ref{bigcor}.    We then obtain chain maps  \[ \Phi_{G_{j,1}^{s},J_{j,s,t}}\co CF^{[c(\rho),d(\rho)]}_{*}(F_{s_j}^{\ep})\to CF^{[c(\rho),d(\rho)]}_{*}(F_{s_{j+1}}^{\ep}), \]\[ \Phi_{G_{j,2}^{s},J'_{j,s,t}}\co CF^{[c(\rho),d(\rho)]}_{*}(F_{s_{j+1}}^{\ep})\to CF^{[c(\rho),d(\rho)]}_{*}(F_{s_j}^{\ep}+\mu)\cong CF^{[c(\rho)+\mu,d(\rho)+\mu]}_{*}(F_{s_j}^{\ep}). \]

Lemma \ref{p2lemma} shows that   $\Phi_{G_{j,2}^{s},J'_{j,s,t}}\circ \Phi_{G_{j,1}^{s},J_{j,s,t}}$ is chain homotopic to \[\Pi^{[c(\rho),d(\rho)]}_{[c(\rho)+\mu,d(\rho)]}\co CF_{*}^{[c(\rho),d(\rho)]}(F_{s_j}^{\ep})\to CF_{*}^{[c(\rho)+\mu,d(\rho)+\mu]}(F_{s_j}^{\ep}).\]  Now by our choice of $\mu$ and the definition of $F_{s}^{\ep}$ we have \[ d(\rho)> \mathcal{A}_{F^{\ep}_{b}}([\gamma_{1,\ell_b}^{-},w_0])\geq\mathcal{A}_{F^{\ep}_{s}}([\gamma_{1,\ell(s)}^{-},w_0])\geq 
 \mathcal{A}_{F^{\ep}_{a}}([\gamma_{1,\ell_a}^{-},w_0])>c(\rho)+\mu,\] so that \[ \Pi^{[c(\rho),d(\rho)]}_{[c(\rho)+\mu,d(\rho)]}[\gamma_{1,\ell(s)}^{-},w_0]=[\gamma_{1,\ell(s)}^{-},w_0].\]
So where $\mathcal{K}_j$ is the chain homotopy produced by Lemma \ref{p2lemma}, we have\begin{align} \label{chtopy} 
 \Phi_{G_{j,2}^{s},J'_{j,s,t}}&\circ \Phi_{G_{j,1}^{s},J_{j,s,t}}([\gamma_{1,\ell(s_j)}^{-},w_0])\\&=[\gamma_{1,\ell(s_j)}^{-},w_0]+\partial_{F_{s_j}^{\ep},J_t}\mathcal{K}_j[\gamma_{1,\ell(s_j)}^{-},w_0]  + \mathcal{K}_j\partial_{F_{s_j}^{\ep},J_t}[\gamma_{1,\ell(s_j)}^{-},w_0].\nonumber\end{align}  But as noted earlier, 
$\partial_{F^{\ep}_{s_j},J'_t}[\gamma_{1,\ell(s_j)}^{-},w_0]=0$ and, for any $[\gamma',w']\in CF^{[c(\rho),d(\rho)]}_{2r+1}(F^{\ep}_{s_j})$, $ \langle \partial_{F^{\ep}_{s_j},J_t}[\gamma',w'],[\gamma_{1,\ell(s)}^{-},w_0]\rangle=0.$  Thus (\ref{chtopy}) shows that
\[ \langle \Phi_{G_{j,2}^{s},J_t}\circ \Phi_{G_{j,1}^{s},J_t}([\gamma_{1,\ell(s_j)}^{-},w_0]),[\gamma_{1,\ell(s_j)}^{-},w_0]\rangle=1.\]
But since for $i=j,j+1$, $\gamma_{1,\ell(s_i)}^{-}$ is the only $1$-periodic orbit $\gamma$ of $X_{F_{s_i}^{\ep}}$ having both $\mu_{F_{s_i}^{\ep}}([\gamma,w_0])=2r$ and $\mathcal{A}_{F_{s_i}^{\ep}}([\gamma,w_0])<d(\rho)$ (the constant orbit at the maximum of $h$ has the right grading but has action approximately zero), Corollary \ref{bigcor} implies that the only possibly-nonzero matrix elements \[ \langle \Phi_{G_{j,1}^{s},J_{j,s,t}}([\gamma_{1,\ell(s_j)},w_0]),[\gamma,w]\rangle \mbox{ or } \langle \Phi_{G_{j,2}^{s},J'_{j,s,t}}[\gamma',w'],[\gamma_{1,\ell(s_j)},w_0]\rangle\]  come from $[\gamma,w]=[\gamma',w']=
[\gamma_{1,\ell(s_{j+1})},w_0]$.  From this it follows that, for some $c_j, c'_j$ with $c_jc'_j=1$, we have \[ 
\Phi_{G_{j,1}^{s},J_t}[\gamma_{1,\ell(s_j)},w_0]=c_j[\gamma_{1,\ell(s_{j+1})},w_0], \quad \Phi_{G_{j,2}^{s},J_t}[\gamma_{1,\ell(s_{j+1})},w_0]=c'_j[\gamma_{1,\ell(s_{j})},w_0].  \] Since we are working over $\mathbb{Z}_2$ we must have $c_j=c'_j=1$.  

From this it follows that \[ \left(\Phi_{G_{N,1}^{s},J_{N,s,t}}\circ\cdots\circ \Phi_{G_{1,1}^{s},J_{1,s,t}}\right)([\gamma_{1,\ell_b}^{-},w_0])=[\gamma_{1,\ell_a}^{-},w_0].\]
Meanwhile,  Corollary \ref{bigcor} together with considerations of action and grading show that, for $\bar{J}_{s,t}\in\mathcal{J}^{reg}(F_{s}^{\ep})$ sufficiently $C^2$-close to the constant path at $\bar{J}$, we have $\Phi_{F_{s}^{\ep},\bar{J}_{s,t}}[\gamma_{1,\ell_b}^{-},w_0]=x[\gamma_{1,\ell_a}^{-},w_0]$ for some $x\in\mathbb{Z}_2$. 
Now another standard gluing and homotopy-of-homotopies argument shows that $\Phi_{G_{N,1}^{s},J_{N,s,t}}\circ\cdots\circ \Phi_{G_{1,1}^{s},J_{1,s,t}}$ is chain homotopic to $\Phi_{F_{s}^{\ep},\bar{J}_{s,t}}$. Since $[\gamma_{1,\ell_b}^{-},w_0]$ is a cycle and 
$[\gamma_{1,\ell_a}^{-},w_0]$ is homologically nontrivial this is possible only if $x=1$, completing the proof of the proposition.

\subsection{Proof of Proposition \ref{p3}}  The foregoing implies that the image of \linebreak[6] $\Psi_{F_{b}^{\ep}}^{H'}\co HF^{[c(\rho),d(\rho)]}_{2r}(F^{\ep}_{b})\to HF^{[c(\rho),d(\rho)]}_{2r}(H')$ is nontrivial for nondegenerate Hamiltonians $H'$ with $F^{\ep}_{b}\leq H'\leq F^{\ep}_{a}$; however, in the spherically irrational case this is no surprise, since the Floer complex of such an  $H'$ has generators in grading $2r$ corresponding to critical points of $h$ on the zero section with nontrivial cappings.  We shall in fact show that, for some $\ast\in \{2r-1,2r\}$, $CF^{[c(\rho),d(\rho)]}_{*}(H')$ must, for suitable nondegenerate $H'$ near $H$, have at least one generator not accounted for by the critical points of $h$.

       Suppose, to get a contradiction, that $X_{H}$ had no $1$-periodic orbits $\gamma$ contained in $\{3\rho-\beta\leq K\leq 3\rho+\beta\}$ having capping discs $w$ with $-2r\leq \Delta([\gamma,w],H)\leq 2m+1$.  As noted earlier, assuming that $\delta>0$ is small enough, the construction of $H$ guarantees that any $1$-periodic orbit of $X_H$ which intersects $\{3\rho-\beta\leq K\leq 3\rho+\beta\}$ is in fact contained in $\{3\rho-\beta<K<3\rho+\beta\}$. The continuity of the Salamon--Zehnder index (and the Arzel\`a-Ascoli theorem) hence implies  that if $\|H'-H\|_{C^2}$ is sufficiently small then $X_{H'}$ also has no $1$-periodic orbits $\gamma$ intersecting $\{3\rho-\beta\leq K\leq 3\rho+\beta\}$ and having capping discs $w$ with $-2r\leq \Delta([\gamma,w],H)\leq 2m+1$.   
       Take for $H'$ a sufficiently small nondegenerate perturbation of $H$, with $H'-H$ supported in $\{3\rho-\beta\leq K\leq 3\rho+\beta\}$ (since the construction of $H$ shows that its constant $1$-periodic orbits are all nondegenerate, while all of its nonconstant $1$-periodic orbits are (for small $\delta$) contained in the interior of $\{3\rho-\beta\leq K\leq 3\rho+\beta\}$,
a standard argument shows that nondegenerate Hamiltonians will form a residual subset of the space of time dependent $H'$ coinciding with $H$ outside $\{3\rho-\beta\leq K\leq 3\rho+\beta\}$).  If $[\gamma',w']$ is any generator of $CF_{*}^{[c(\rho),d(\rho)]}(H')$, then using  Equation 2.2 of \cite{GG2} (and the fact that the Maslov index $\mu_{H'}$ is related to the Conley--Zehnder index $\mu_{CZ}$ by $\mu_{CZ}=m+r-\mu_{H'}$) we have $0\leq \Delta([\gamma,w],H')+\mu_{H'}([\gamma,w])\leq 2(m+r)$.  So our assumption implies that if $[\gamma',w']$ has $\mu_{H'}([\gamma',w'])\in \{2r-1,2r\}$ then $\gamma'$ does not intersect $\{3\rho-\beta\leq K\leq 3\rho+\beta\}$.  
Now $H'$ coincides with $H$  outside $\{3\rho-\beta\leq K\leq 3\rho+\beta\}$, and so the only generators $[\gamma',w']$ for  $CF_{*}^{[c(\rho),d(\rho)]}(H')$ in gradings $2r-1$ and $2r$ are those where $\gamma'$ is the constant orbit at a critical point of $h\co M\to \mathbb{R}$ on the zero section $M\subset E(R)$.

Now let $\bar{f}\co [0,\infty)\to [0,\infty)$ be a strictly increasing function with $0\leq\bar{f}\leq f_b$ everywhere, and $0<\bar{f}'(s)\leq 2\pi$ for all $s$, with $\bar{f}'(s)=2\pi$ for $s\geq \frac{30 \rho}{C_1}$.  Let $\bar{F}=\bar{f}\circ L+\delta h\circ \pi$.  The only $1$-periodic orbits of $X_{\bar{F}}$ are then the constant orbits at the critical points of $h$ on the zero section $M\subset E(R)$.  
In particular, \emph{in degrees $*=2r-1,2r$, we have $CF_{*}^{[c(r),d(r)]}(H')=CF_{*}^{[c(r),d(r)]}(\bar{F})$ as $\mathbb{Z}_2$-modules}.

We have $\bar{F}\leq F_{b}^{\ep}\leq H'$, so for generic $J_t$ close to $\bar{J}$ let $(H^{s}_{1},J_{s,t})$ be a regular monotone homotopy from $(\bar{F},J_t)$ to $(F_{b}^{\ep},J_t)$ and let $(H^{s}_{2},J'_{s,t})$ be a regular monotone homotopy from   $(F_{b}^{\ep},J_t)$ to $(H',J_t)$.  (For definiteness, set $H^{s}_{1}=(1-\chi(s))\bar{F}+\chi(s)F_{b}^{\ep}$ and $H^{s}_{2}=(1-\chi(s))F_{b}^{\ep}+\chi(s)H'$, where $\chi\co\mathbb{R}\to [0,1]$ is a  smooth monotone increasing function with $\chi(s)=0$ for $s<-1$ and $\chi(s)=1$ for $s>1$).  We then have chain maps \[ 
\xymatrix{ \oplus_{k=2r-1}^{2r}CF_{k}^{[c(\rho),d(\rho)]}(\bar{F})\ar[r]^{\Phi_{H^{s}_{1},J_{s,t}}}&\oplus_{k=2r-1}^{2r}CF_{k}^{[c(\rho),d(\rho)]}(F_{b}^{\ep})\ar[r]^{\Phi_{H^{s}_{2},J'_{s,t}}}& \oplus_{k=2r-1}^{2r}CF_{k}^{[c(\rho),d(\rho)]}(H')}.\]  Now by construction the critical points $p_j$ of $h$ on $M$ are nondegenerate critical points of each Hamiltonian $H^{s}_{1}$, $H^{s}_{2}$, and the (small) choices of $\delta$ and $\eta$ ensure that the norms of the Hessians of $H^{s}_{1}$, $H^{s}_{2}$ at these critical points are smaller than $1$.  Further, we have $H^{s}_{1}(p_j)=H^{s}_{2}(p_j)=\delta h(p_j)$ for each $s$ and $j$.  Hence the constant paths $\gamma_{j,0}$ at $p_j$ satisfy the hypotheses of Proposition \ref{ipluslower} for the monotone homotopies $(H^{s}_{1},J_{s,t})$ and $(H^{s}_{2},J'_{s,t})$.  But $\oplus_{k=2r-1}^{2r}CF_{k}^{[c(r),d(r)]}(H')$ and $\oplus_{k=2r-1}^{2r}CF_{k}^{[c(r),d(r)]}(\bar{F})$ are both just equal to the span over $\mathbb{Z}_2$ of the various generators $[\gamma_{j,0},w]$ having appropriate action and grading. Therefore, Proposition \ref{ipluslower} shows that 
  the map \[  \Phi_{H^{s}_{2},J'_{s,t}}\circ \Phi_{H^{s}_{1},J_{s,t}}\co
\oplus_{k=2r-1}^{2r}CF_{k}^{[c(\rho),d(\rho)]}(\bar{F})\to \oplus_{k=2r-1}^{2r}CF_{2r}^{[c(\rho),d(\rho)]}(H')\] has the form $I+D$ where $I$ is the identity and $D$ has the property that, for some $\mu>0$, we have $\mathcal{A}_{\bar{F}}([\gamma,w])\geq \mathcal{A}_{\bar{F}}([\gamma',w'])+\mu$ whenever $\langle D[\gamma,w],[\gamma',w']\rangle\neq 0$.  Now any such map is invertible, with inverse given by $A=\sum_{j=0}^{\infty}(-D)^{j}$  (this is a finite sum, since $D^j=0$ if $j\mu>d(\rho)-c(\rho)$).  For $x\in CF_{2r}^{[c(\rho),d(\rho)]}(H')$ we have \[ A(\partial_{H',J_t}x)=A(\partial_{H',J_t}\Phi_{H^{s}_{2},J'_{s,t}}\circ \Phi_{H^{s}_{1},J_{s,t}}Ax)=A(\Phi_{H^{s}_{2},J'_{s,t}}\circ \Phi_{H^{s}_{1},J_{s,t}}\partial_{\bar{F},J_t}Ax)=
\partial_{\bar{F},J_t}Ax.\]

For any nondegenerate Hamiltonian $G$, define $\mathcal{L}_G\co CF_{*}^{[c(\rho),d(\rho)]}(G)\to\mathbb{R}\cup\{-\infty\}$ by \[ \mathcal{L}_G\left(\sum c_{[\gamma,w]}[\gamma,w]\right)=\max\left\{\mathcal{A}_G([\gamma,w])|c_{[\gamma,w]}\neq 0\right\}\] (where the maximum of the empty set is defined to be $-\infty$).  Note that if $(G^s,\bar{J}_{s,t})$ is a regular monotone homotopy from $(G^-,J^{-}_{t})$ to $(G^+,J^{+}_{t})$, we have $\mathcal{L}_{G^+}(\Phi_{G^s,\bar{J}_{s,t}}(x))\leq \mathcal{L}_{G^-}(x)$ for all $x\in CF_{*}^{[c(\rho),d(\rho)]}(G^-)$.  Note also that the map $A=\sum_{j=0}^{\infty}(-D)^j$ of the previous paragraph satisfies $\mathcal{L}_{\bar{F}}(Ax)=\mathcal{L}_{H'}(x)$ for all $x\in \oplus_{k=2r-1}^{2r}CF_{k}^{[c(\rho),d(\rho)]}(H')$.

  Write \[ c=\Phi_{H^{s}_{2},J'_{s,t}}([\gamma_{1,\ell_b}^{-},w_0])\in CF_{2r}^{[c(\rho),d(\rho)]}(H')\] and \[ c'=\Phi_{H^{s}_{1},J_{s,t}}(A(c))\in CF_{2r}^{[c(\rho),d(\rho)]}(F^{\ep}_{b}).\]  Then \[ \Phi_{H^{s}_{2},J'_{s,t}}\left([\gamma_{1,\ell_b}^{-},w_0]-c'\right)=c-c=0.\]  Since $A$ and the various chain maps $\Phi$ intertwine the relevant Floer boundary operators, we have $\partial_{F_{b}^{\ep},J_t}c'=0$.
  
Meanwhile, since $[\gamma_{1,\ell_b}^{-},w_0]$ is not a degree-$2r$ generator for $CF_{2r}^{[c(\rho),d(\rho)]}(H')$, any solution $u$ of (\ref{htopyeqn}) contributing to a matrix element $\langle\Phi_{H^{s}_{2},J'_{s,t}}([\gamma_{1,\ell_b}^{-},w_0]),[\gamma,w]\rangle$ must have $\int_{\mathbb{R}\times (\mathbb{R}/\mathbb{Z})}\left|\frac{\partial u}{\partial s}\right|^2 dsdt>0$.  In view of this, (\ref{htopyenergy}) implies that \[ \mathcal{L}_{H'}(c)<\mathcal{L}_{F_{b}^{\ep}}([\gamma_{1,\ell_b}^{-},w_0]).\]  Hence also  
\[ \mathcal{L}_{F_{b}^{\ep}}(c')<\mathcal{L}_{F_{b}^{\ep}}([\gamma_{1,\ell_b}^{-},w_0]).\]

 
But then  since $[\gamma_{1,\ell_b}^{-},w_0]$ is the only generator of $CF^{[c(\rho),d(\rho)]}_{2r}(F_{b}^{\ep})$ with the fiberwise capping (all other fiberwise-capped generators have Maslov indices unequal to $2r$ or actions outside $[c(\rho),d(\rho)]$, as we have seen earlier), we have \begin{equation}\label{cnontriv} [\gamma_{1,\ell_b}^{-},w_0]-c'=[\gamma_{1,\ell_b}^{-},w_0]-\sum_{i}c'_i[\gamma_i,w_0\#A_i] \end{equation} where $c'_i\in\mathbb{Z}_2$ and each $A_i\neq 0$.  

Now where $(H^{s}_{3},J''_{s,t})$ is a monotone homotopy from $(H',J_t)$ to $(F_{a}^{\ep},J''_t)$ (for some generic $J''_t$ near $\bar{J}$), we have that $\Phi_{H^{s}_{3},J''_{s,t}}\circ \Phi_{H^{s}_{2},J'_{s,t}}$ induces the same map $\Psi_{F_{b}^{\ep}}^{F_{a}^{\ep}}$ on homology as does $\Phi_{F_{s}^{\ep},\bar{J}_{s,t}}$ (where $(F_{s}^{\ep},\bar{J}_{s,t})$ is as in  Proposition \ref{p2}).  By (\ref{cnontriv}), Corollary \ref{bigcor}, and Proposition \ref{p2}, we have  \[ \Phi_{F_{s}^{\ep},\bar{J}_{s,t}}\left([\gamma_{1,\ell_b}^{-},w_0]-c'\right)=  [\gamma_{1,\ell_a}^{-},w_0]+c''\] where $c''$ has form $\sum_i c''_i[\gamma''_i,w_0\#A_i]$ and every $A_i\neq 0$.  Now if the chain on the right hand side above were nullhomologous, there would need to be a nonzero matrix element $\langle \partial_{F_{a}^{\ep},J_t}[\gamma,w], [\gamma_{1,\ell_a}^{-},w_0]\rangle$, whereas we have shown in  Proposition \ref{p1} that all such matrix elements are zero.  So $\Phi_{F_{s}^{\ep},\bar{J}_{s,t}}\left([\gamma_{1,\ell_b}^{-},w_0]-c'\right)$ is homologically nontrivial, which \textbf{contradicts} the facts that $\Phi_{F_{s}^{\ep},\bar{J}_{s,t}}$ induces the map $\Psi_{F_{b}^{\ep}}^{F_{a}^{\ep}}=\Psi_{H'}^{F_{a}^{\ep}}\circ \Psi_{F_{b}^{\ep}}^{H'}$ on homology, and that \[ \Psi_{F_{b}^{\ep}}^{H'}([[\gamma_{1,\ell_b}^{-},w_0]-c'])=[\Phi_{H^{s}_{2},J'_{s,t}}([\gamma_{1,\ell_b}^{-},w_0]-c')]=0.\]

\section{Appendix: Background on local Floer homology for clean intersections} 
Let $(P,\Omega)$ be an arbitrary symplectic manifold, and suppose that $H_0\co P\to \mathbb{R}$ is an autonomous Hamiltonian on $P$ inducing a flow $\phi_{H_0}^{t}$ which has the property that fixed point set of $\phi_{H_0}^{1}$ has a connected component $N\subset Fix(\phi_{H_0})$ such that \begin{itemize} \item $N$ is a compact submanifold of $P$, \item $H_0|_N$ is constant, \item there is a $\mathbb{R}/\mathbb{Z}$-action on a neighborhood of $N$, which preserves $N$, has orbits which are contractible in $P$, and has the property that for each $x\in N$ we have $t\cdot x=\phi_{H_0}^{t}(x)$;  and \item for each $x\in N$, $\ker(Id-(d\phi_{H_0})_x)=T_xN$.\end{itemize}    Denote by $\mathcal{L}_0P$ the space of contractible loops in $P$, $\mathcal{N}\subset \mathcal{L}_0P$ the subset consisting of $(\mathbb{R}/\mathbb{Z})$-orbits through  points of $N$, and $\mathcal{U}\supset \mathcal{N}$ the closure of a neighborhood of $\mathcal{N}$, which should be taken to be small in a sense to be specified presently.  Let $U\subset P$ be a tubular neighborhood of $N$, identified with the disc normal bundle to $N$ with projection $\tau\co U\to N$, and taken small enough that the only fixed points of $\phi_{H_0}^{1}$ in $U$ are the points of $N$.  $\mathcal{U}$ is then chosen small enough that every $\gamma\in \mathcal{U}$ has the properties that: (i) $\gamma(t)\in U$ for all $t$, and (ii) with respect  to a $(\mathbb{R}/\mathbb{Z})$-invariant Riemannian metric on $N$, the diameter of the loop  $t\mapsto (\phi_{H_0}^{t})^{-1}(\tau(\gamma(t)))$ is to be less than the injectivity radius of $N$.  (Thus, every $\gamma\in \mathcal{U}$ has $\tau\circ \gamma$ close to a $(\mathbb{R}/\mathbb{Z})$-orbit).

Following [5, p. 581], for a $(\mathbb{R}/\mathbb{Z})$-parametrized family of $\Omega$-tame almost complex structures  $J_t$ and a Hamiltonian $H\co (\mathbb{R}/\mathbb{Z})\times P\to \mathbb{R}$ define \[ \mathcal{M}_{J_t,H}(\mathcal{U})=\left\{u\co\mathbb{R}\times S^1\to P\left|\begin{array}{l}\int_{\mathbb{R}\times S^1}\left|\frac{\partial u}{\partial s}\right|^{2}_{J_t}dsdt<\infty,\, \frac{\partial u}{\partial s}+J_t\left(\frac{\partial u}{\partial t}-X_H\right)=0\\ \mbox{ and }u(s,\cdot)\in \mathcal{U}\mbox{ for all }s\end{array}\right.\right\},\] and \[ \mathcal{S}_{J_t,H}(\mathcal{U})=\{\gamma\in \mathcal{U}|  \gamma=u(0,\cdot)\mbox{ for some }u\in \mathcal{M}_{J_t,H}(\mathcal{U})\}.\]

\begin{prop}\label{invar} We have $\mathcal{S}_{J_t,H_0}(\mathcal{U})=\mathcal{N}$.\end{prop}

\begin{proof} Define $\widetilde{\mathcal{L}_0P}$ to be the space of equivalence classes $[\gamma,w]$ where $\gamma$ is a contractible loop in $P$ and $w\co D^2\to P$ has $w|_{\partial D^2}=\gamma$, with $[\gamma,w]=[\gamma,w']$ if and only if $\int_{D^2}w^*\Omega=\int_{D^2}w'^*\Omega$.  Elements of $\mathcal{M}_{J_t,H_0}(\mathcal{U})$ are then formal negative gradient flow lines for the functional \[ \mathcal{A}_{H_0}([\gamma,w])=-\int_{D^2}w^*\Omega-\int_{0}^{1}H_0(\gamma(t))dt\] on $\widetilde{\mathcal{L}_0P}$.  Since the only fixed points of $\phi_{H_0}^{1}$ in $U$ are on $N$, any $u\in \mathcal{M}_{J_t,H_0}(\mathcal{U})$ is asymptotic to $\mathbb{R}/\mathbb{Z}$-orbits $\gamma^{\pm}=u(\pm\infty,\cdot)$ in $N$.  The extended map $\bar{u}\co [-\infty,\infty]\times (\mathbb{R}/\mathbb{Z})\to U$ is then homotopic rel boundary to $\tau\circ \bar{u}\co [-\infty,\infty]\times (\mathbb{R}/\mathbb{Z})\to N$, and our definition of $\mathcal{U}$ then ensures that, using the exponential map of the metric on $N$, we can obtain a homotopy rel boundary from $\tau\circ \bar{u}$ to the map $\bar{v}\co (s,t)\mapsto \phi_{H}^{t}(\tau(u(s,0)))$.  Hence $\int_{[-\infty,\infty]\times(\mathbb{R}/\mathbb{Z})}\bar{u}^*\Omega=\int_{[-\infty,\infty]\times(\mathbb{R}/\mathbb{Z})}\bar{v}^*\Omega$.  But we have $\frac{\partial \bar{v}}{\partial t}=X_{H_0}$ everywhere, and the image of $\bar{v}$ is contained in $N$, on which $H_0$ is constant, so that \[ \Omega\left(\frac{\partial \bar{v}}{\partial s},\frac{\partial \bar{v}}{\partial t}\right)=-dH_0\left(\frac{\partial \bar{v}}{\partial s}\right)=0,\] and thus \[ \int_{\mathbb{R}\times S^1}u^*\Omega=0.\]  Now if $u$ is a flowline from $[\gamma^-,w^-]$ to $[\gamma^+,w^+]$, we have \[ \mathcal{A}_{H_0}([\gamma^-,w^-])-\mathcal{A}_{H_0}([\gamma^+,w^+])=\int_{\mathbb{R}\times S^1}u^*\Omega+\int_{0}^{1}(H_0(\gamma^+(t))-H_0(\gamma^-(t)))dt.\]  We've just shown that the first term on the right hand side above is zero, while the second term is zero since $\gamma^{\pm}$ are loops in $N$ and $H_0$ is constant 
 on $N$.

Using the familiar formulas \begin{align*} \mathcal{A}_{H_0}([\gamma^-,w^-])-\mathcal{A}_{H_0}([\gamma^+,w^+])&=\int_{\mathbb{R}\times S^1}\left|\frac{\partial u}{\partial s}\right|^{2}_{J_t}dsdt\\&=\int_{\mathbb{R}\times S^1}\left|\frac{\partial u}{\partial t}-X_{H_0}(u(s,t))\right|^{2}_{J_t}dsdt,\end{align*} this shows that our arbitrary $u\in \mathcal{M}_{J,H}(\mathcal{U})$ (which is known \emph{a priori} to be $C^1$ by standard arguments) must have $\frac{\partial u}{\partial t}=X_{H_0}(u(s,t))$ everywhere, in view of which $u(0,\cdot)$ is a $1$-periodic orbit of $\phi_{H_0}^{t}$.  Since the only such orbits belonging to $\mathcal{U}$ are the elements of $\mathcal{N}$, this proves the proposition.
\end{proof}

If $\ep>0$, define \[ \mathcal{M}_{J_t,H}^{\ep}(\mathcal{U})=\{u\in  \mathcal{M}_{J_t,H}(\mathcal{U})|\int_{\mathbb{R}\times S^1}\left|\frac{\partial u}{\partial s}\right|^{2}_{J_t}<\ep\},\] and\[ \mathcal{S}^{\ep}_{J_t,H}(\mathcal{U})=\{\gamma\in \mathcal{U}|  \gamma=u(0,\cdot)\mbox{ for some }u\in \mathcal{M}_{J_t,H}^{\ep}(\mathcal{U})\} \]

Let \[ \hbar\{J_t\}=\inf\{\int_{S^2}u^*\Omega| u\co S^2\to U, \delbar_{J_t}u=0 \mbox{ for some }t, u\mbox{ is nonconstant}\}.\]
Of course, Gromov compactness implies that $\hbar\{J_t\}>0$.  

\begin{prop} \label{appprop} (Cf. [5, Theorem 3])  Let $0<\delta<\hbar\{J_t\}/2$.  Then there is $\eta>0$ with the following properties.  If $\|H-H_0\|_{C^2}<\eta$, and if $J'_t$ is a $(\mathbb{R}/\mathbb{Z})$-parametrized family of almost complex structures with $\|J'_t-J_t\|_{C^2}<\eta$ for each $t$,  then whenever $\ep<\hbar\{J_t\}$ $\mathcal{S}^{\ep}_{J'_t,H}(\mathcal{U})$ is contained in the interior of $\mathcal{U}$.  Further, for $\|H-H_0\|_{C^2}<\eta$ and $\|J'_t-J_t\|_{C^2}<\eta$, $\mathcal{M}^{\ep}_{J'_t,H}(\mathcal{U})$ is independent of $\ep$ for all choices of $\ep\in [\delta,\hbar\{J_t\})$.
\end{prop}

\begin{proof} Suppose that $H_n\to H_0$ and $J_{t,n}\to J_t$ in $C^2$-norm, and that $u_n\in \mathcal{M}^{\ep}_{J_{t,n},H_n}(\mathcal{U})$ where $\ep<\hbar\{J_t\}$.  Suppose that there were $s_n\in\mathbb{R}$ having the property that $u_n(s_n,\cdot)\notin \mathcal{U}^{\circ}$.  Gromov compactness applied to the maps $(s,t)\mapsto u_n(s-s_n,t)$ then implies that the $u_n(s_n,\cdot)$ converge to some element of $\mathcal{S}_{J_t,H_0}(\mathcal{U})$ (bubbling is precluded because the $u_n$ have energy less than the minimal energy of any $J_t$-holomorphic sphere in $U$).  But we have shown that  $\mathcal{S}_{J_t,H_0}(\mathcal{U})=\mathcal{N}$, so since $\mathcal{U}^{\circ}$ is a neighborhood of $\mathcal{N}$ this is a contradiction, which proves the first statement of the proposition.

For  $\|H-H_0\|_{C^2}$ small, any one-periodic orbit $\gamma$ of $X_{H}$ in $\mathcal{U}$ is close to a $1$-periodic orbit of $X_{H_0}$, \emph{i.e.}, to a loop that belongs to $\mathcal{N}$.  As such, we can choose $\eta$ small enough that (in addition to the first statement of the proposition holding), for $\|H-H_0\|_{C^2}<\eta$, whenever $\gamma\in\mathcal{U}$ is a $1$-periodic orbit of $X_{H}$ there exists a map $u_{\gamma}\co [0,1]\times(\mathbb{R}/\mathbb{Z})\to U$ such that $u_{\gamma}(0,\cdot)=\gamma$,  $u_{\gamma}(1,\cdot)\in\mathcal{N}$, and $\left|\int_{[0,1]\times(\mathbb{R}/\mathbb{Z})}u_{\gamma}^{*}\Omega\right|<\delta/4$.  Let $C$ denote the cylinder obtained by attaching copies $C_-$ and $C_+$ of $[0,1]\times(\mathbb{R}/\mathbb{Z})$ to, respectively, the left and right ends of $[-\infty,\infty]\times(\mathbb{R}/\mathbb{Z})$.  
Given $u\in\mathcal{M}^{\ep}_{J'_t,H}(\mathcal{U})$, such that $u(\pm\infty,\cdot)\to\gamma_{\pm}$, define 
$ \tilde{u}\co  C\to U$ by setting $\tilde{u}(s,t)$ equal to $u_{\gamma_-}(1-s,t)$ on $C_-$, to $u(s,t)$ on $[-\infty,\infty]\times(\mathbb{R}/\mathbb{Z})$, and to $u_{\gamma_+}(s,t)$ on $C_+$.  Since $\tilde{u}$ maps the boundary loops of $C$ to loops belonging to $\mathcal{N}$, an argument identical to that used in the proof of Proposition \ref{invar} establishes that $\int_C \tilde{u}^*\Omega=0$.  Hence \[ \left| \int_{\mathbb{R}\times(\mathbb{R}/\mathbb{Z})}u^*\Omega\right|\leq \left|\int_{[0,1]\times(\mathbb{R}/\mathbb{Z})}u_{\gamma_-}^{*}\Omega\right|+
\left|\int_{[0,1]\times(\mathbb{R}/\mathbb{Z})}u_{\gamma_-}^{*}\Omega\right|<\delta/2\] whenever $u\in \mathcal{M}_{J'_t,H}^{\ep}(\mathcal{U})$ with $\|H-H_0\|_{C^2}<\eta$ and $\|J'_t-J_t\|_{C^2}<\eta$.  

Now since $H_0$ is constant on $N$ and since the $1$-periodic orbits of $H$ in $\mathcal{U}$ approach $N$ as $H$ approaches $H_0$, by decreasing $\eta$ if necessary we can arrange that, whenever $\gamma_{\pm}$ are $1$-periodic orbits of $X_H$ in $\mathcal{U}$ and $\|H-H_0\|_{C^2}<\eta$, we have $|H(t,\gamma_+(t))-H(t,\gamma_-(t))|<\delta/2$ for all $t$.  So if 
$u\in \mathcal{M}_{J_t,H}^{\ep}(\mathcal{U})$ is as above, then (for any capping $w$ of $\gamma_-$) \begin{align*} \int_{\mathbb{R}\times(\mathbb{R}/\mathbb{Z})}\left|\frac{\partial u}{\partial s}\right|^{2}_{J'_t}dsdt&=\mathcal{A}_{H}([\gamma_-,w])-\mathcal{A}_H([\gamma_+,w\#u])\\&=\int_{\mathbb{R}\times(\mathbb{R}/\mathbb{Z})}u^*\Omega+\int_{\mathbb{R}/\mathbb{Z}}(H(t,\gamma_+(t))-H(t,\gamma_-(t)))dt<\delta.\end{align*}  This proves that $\mathcal{M}^{\ep}_{J'_t,H}(\mathcal{U})$  coincides with $\mathcal{M}^{\delta}_{J'_t,H}(\mathcal{U})$ 
whenever $\ep\in [\delta,\hbar\{J_t\})$.
\end{proof}
  
Given this, choosing any $\delta<\hbar\{J_t\}/2$, standard arguments allow one to define the local Floer complex $CF^{loc}_{*}(H,H_0,\mathcal{U})$ over $\mathbb{Z}_2$ by choosing a generic Hamiltonian $H$ with $\|H-H_0\|_{C^2}<\eta$ and using as generators those periodic orbits of $X_{H}$ which lie in $\mathcal{U}$.  The boundary operator $\partial^{loc}_{H}$ may be \emph{a priori} defined to count elements of $\mathcal{M}^{2\delta}_{J'_t,H}(\mathcal{U})$ for generic $J'_t$ in the usual fashion; the fact that any such element in fact belongs to $\mathcal{M}^{\delta}_{J'_t,H}(\mathcal{U})$ implies that one has $(\partial^{loc}_{H})^2=0$.  
As in [5, Theorem 4] the resulting homology $HF^{loc}_{*}(H_0,\mathcal{U})$  is independent of the choice of $H$ (and of $\delta$) provided that $H$ is appropriately $C^2$-close to $H_0$.  

Now all of the above may be recast in terms of local Lagrangian Floer homology for the intersection of the diagonal $\Delta$ with the graph $gr(\phi_{H_0})$ in $(P\times P,\Omega\oplus(-\Omega))$, with the role of $\mathcal{N}$ played by the constant paths at points of $N'=\{(n,n)|n\in N\}$ and the role of $\mathcal{U}$ played by paths from $\Delta$ to    $gr(\phi_{H_0})$ which remain near $N'$ and are close to being constant (here one also needs to choose the parameter $\delta$ of Proposition \ref{appprop} small enough as to preclude the bubbling off of pseudoholomorphic discs).  Note that our assumptions on $H_0$ show that $N'$ is a clean intersection between $\Delta$ and $gr(\phi_{H_0})$, which is to say that $N'$ is a compact connected component of $\Delta\cap gr(\phi_{H_0})$ satisfying $T_xN=(T_x\Delta)\oplus (T_x gr(\phi_{H_0}))$ for all $x\in N$. On the one hand, the usual correspondence between Hamiltonian Floer homology and the Lagrangian Floer homology of graphs shows that $HF^{loc}_{*}(H_0,\mathcal{U})$ is isomorphic to the resulting local Lagrangian Floer homology.  On the other hand, local Lagrangian Floer homology in the context 
 of clean intersections has been analyzed in \cite{P}, in which it is shown (as part 3 of Theorem 3.4.11) that the local Floer homology is isomorphic to the Morse homology of the intersection $N'$.  So since $N'$ is diffeomorphic to $N$, we conclude that:

\begin{theorem}(\cite{P})\label{morsebott} $HF^{loc}_{*}(H_0,\mathcal{U})=H_*(N,\mathbb{Z}_2)$.
\end{theorem}
\begin{remark} Since, for the generators of $CF^{loc}_{*}(H,H_0,\mathcal{U})$, we take Hamiltonian periodic orbits without any ``capping data,'' the grading of the local Floer chain complex should be understood as a \emph{relative} grading by the cyclic group $\mathbb{Z}/\Gamma$ where $\Gamma$ is twice the minimal Chern number of $P$.   In particular, Theorem \ref{morsebott} is an isomorphism of relatively $\mathbb{Z}$-graded groups when  $c_1(TP)$ vanishes.

\end{remark}

\end{document}